\documentclass[11pt,english]{article}
\usepackage[T1]{fontenc}
\usepackage{amsmath}

\DeclareMathOperator*{\argmin}{arg\,min}
\usepackage{amsthm}
\usepackage{amssymb}
\usepackage{appendix}
\usepackage[colorinlistoftodos]{todonotes}
\usepackage[nobysame]{amsrefs}
\usepackage{breqn}
\usepackage{dsfont}
\usepackage{thmtools,thm-restate}
\usepackage{hyperref}
\usepackage{algorithm,algorithmic}

\makeatletter
\theoremstyle{plain}
\newtheorem{theorem}{\protect\theoremname}[section]
\theoremstyle{plain}

\theoremstyle{plain}
\newtheorem{proposition}{\protect\propositionname}[section]
\theoremstyle{plain}
\newtheorem{corollary}{\protect\corname}[section]
\theoremstyle{plain}

\theoremstyle{definition}
\newtheorem{definition}{\protect\definitionname}[section]
\theoremstyle{definition}

\newtheorem{assumption}{Assumption}[section]
\ifx\proof\undefined
\newenvironment{proof}[1][\protect\proofname]{\par
	\normalfont\topsep6\p@\@plus6\p@\relax
	\trivlist
	\itemindent\parindent
	\item[\hskip\labelsep\scshape #1]\ignorespaces
}{%
	\endtrivlist\@endpefalse
}
\providecommand{\proofname}{Proof}
\fi

\usepackage{fullpage}

\usepackage{times}

\makeatother

\usepackage{babel}
\providecommand{\definitionname}{Definition}
\providecommand{\theoremname}{Theorem}
\providecommand{\lemmaname}{Lemma}
\providecommand{\propositionname}{Proposition}
\providecommand{\propertyname}{Property}
\providecommand{\examplename}{Example}
\providecommand{\corname}{Corollary}

\usepackage{pgfgantt}
\usepackage{bigints}
\usepackage{listings}
\usepackage{verbatim}

\usepackage{float}
\usepackage{subcaption}
\usepackage{fullpage}
\usepackage{authblk}
\newcommand{\N}{\ensuremath{\mathbb{N}}}

\definecolor{mygreen}{RGB}{28,172,0}
\definecolor{mylilas}{RGB}{170,55,241}
\usepackage{color}

\begin{document}
\title{To Control or not to Control}
%
%

\author[1]{Odysseas Kanavetas}
\author[2]{Camiel M.P. Koopmans}
 \author[2]{Floske M. Spieksma}
 \affil[1]{ Department of Mathematics, University of Patras, Greece}
 \affil[2]{ Mathematisch Instituut, Universiteit Leiden, the Netherlands}
 
%
\pagenumbering{gobble}
\maketitle              
\begin{abstract}
We introduce a model with control vacations instead of standard queueing control systems with permanent control.
        The researched model is an M/M/1 queue with temporary periods of service rate control with two available service rates. After a control period of exponentially distributed length, a control vacation is initiated during which a fixed service rate $\mu$ is used. The start of the next control period needs to be scheduled directly at a certain cost.
 We will use the Markov Decision Process from Kanavetas et al.~\cite{Koopmans} to find a sufficient condition that ensures that the average expected cost can be reduced w.r.t. the model that only uses the fixed service rate.
Under this condition we will use properties of this related process to construct a cost reducing policy. The process with control vacations under specific policies induces a renewal reward process.
 We use a Tauberian theorem to relate the average expected cost of this renewal reward process to a vanishing discount method and analytically determine a lower bound of the average expected cost reduction.
Finally, we study the actual attained average cost reduction for these policies through simulation.

  

\end{abstract}

\newpage
\pagenumbering{arabic}
\section{Introduction}\label{Introduction}

In this paper, we consider an M/M/1 queue where a
baseline of service is guaranteed for the system, and availability of an alternative/additional server/machine/processor is restricted by vacations or breakdowns. As a result, there is an availability of service rate control until vacation/breakdown of the additional service mode.

An example of this idea can be found on highways with rush-hour lanes as in the Netherlands. When traffic intensity is high, rush hour lanes are opened to increase traffic flow and reduce congestion. When there are sensor disruptions, the rush hour lanes remain closed for safety reasons~\cite{spitsstrookrws}. This has happened on multiple occasions (see \cite{traffic3,traffic1,traffic2}). In this framework there is a natural trade-off between safety, traffic congestion, and costs for sensor repairs.

Vacation in queueing models has been introduced by Levi and Yechiali~\cite{levy1975utilization} in 1975. Here, a server vacation can occur as the server, when idle, can be utilised for service to a secondary system. 
Since then, extensive research has been conducted on vacation models, these vacations have been restricted to server vacations, as substantiated by the following works~\cite{doshi1986queueing,takagi1991queueing,tian2006vacation,ke2010recent}.

As Doshi~\cite{doshi1986queueing} describes, there are multiple reasons for vacation models to arise including (machine) breakdowns, maintenance, or specific queueing disciplines. Doshi also mentions a possible connection to priority models, which in a similar context to this paper can occur with a low priority queue and a high priority queue. The fastest service mode is ensured to the high priority queue, such that  this fastest service speed is only available for the low priority queue when the high priority queue is idle. 

The vacations in the rush-hour lane example and the implementation of the vacations in this article suit the type induced by machine breakdowns. At a breakdown, only the baseline of service is available until repair of the machine, which comes at a cost. This leaves the repair time of the machine to be scheduled with the resulting effects on the model.

The objective of the service rate control is to minimise all costs. One component of this cost is the service rate cost, which is either the direct costs for the speed of service or an indirect cost through a difference in the service quality (cf.~\cite{gryna,rust}). The second component of the cost are the \textit{holding costs} for the amount of customers in the system. These costs are incurred due to customer dissatisfaction/costs of waiting for the customers or exhaustion of resources by the customers. Furthermore, repair costs are incurred to regain service rate control.

The subjects of research we consider for this model are as follows
\begin{itemize}
    \item The existence and construction of policies that reduce the average expected cost w.r.t. the baseline of service.
    \item The size of the attained average expected cost reduction w.r.t. the baseline of service. 
\end{itemize}

In order to tackle the research objectives, research from Kanavetas et al.~\cite{Koopmans} into the saved cost from a single period of service rate control will prove to be valuable. Both the amount of cost saved, and optimal policies will be of use for this article.

The model we will consider with the scheduling of repair times bears similarities with the subject of controlled server vacations. For research on controlled server vacations we refer to~\cite{convac1,convac2,convac3}. 


A different viewpoint of this type of model would be through \textit{control vacations} for queueing models, where a certain method of queueing control is unavailable for temporary time. Through this viewpoint, numerous research questions and models (with service rate control, admission control and/or priority queues) can be constructed for further research.

\subsection*{Problem description}

The model studied in this paper is an M/M/1 queue with the option to pay $c_r$ for the repair/instalment of a sensor that observes the length of the queue and guides service rate control. 
 Customers arrive to the queue according to a Poisson$(\lambda)$ process.
We assume exponentially distributed service times.
 Additionally, the sensor has an exp($\beta$) distributed lifetime.
 During this period, the system is controlled by selecting one of two service rates $\mu_1$ or $\mu_2$ depending on the queue length. 
 At the instant the sensor breaks down, there is the option to an instantaneous repair, or to schedule a future repair time. Note that, before sensor repair,  no new information is gained that would affect the earlier choice of the repair time.
 During sensor failure, the server operates at a baseline rate $\mu\in \{\mu_1,\mu_2\}$ with $\mu>\lambda$ to ensure stability.



 The state of the system is denoted by the number of customers in the system and whether the sensor is active or not. If the sensor is not active then the state also includes the time until the next scheduled sensor repair.
 Thus, the state space is given by
$$S=\{ i,(i,0),(i,1) : i\in \N_0 \}\cup \{ (i,0,t) :  i\in \N_0,t\in\mathbb{R}_{\geq 0}\} .$$
In these states, $i$ denotes the number of customers in the system. The sensor is active in 
states $(i,1)\in S$.
The sensor is inactive in states $(i,0)\in S$ and $(i,0,t)\in S$, where in the latter, $t$ denotes the remaining time until repair. In state $(i,0)$ there will not be any future sensor repair.
Finally, state $i\in S$ is used for initialisation of the process with a broken sensor, where the repair time will be chosen instantaneously. 

The state changes upon arrivals, departures, sensor failures, or repairs.
We restrict the decision of the time until repair to deterministically depend on the amount of customers in the system at the moment the sensor breaks down.
From a modelling perspective this leads to action space 
 $$\mathcal{A}(i,1)=\{(a,t) : a\in \{\mu_1,\mu_2\},\ t\in \mathbb{R}_{\geq 0}\cup\{\infty\}\},\ \ \text{for } (i,1)\in S,$$  where $a$ denotes the chosen service rate and $t$ denotes the scheduled repair time if the sensor would break down in this state (where $t=\infty$ would mean that the sensor will never be repaired again). In states $(i,0)\in S $ and $(i,0,t)\in S$, service rate $\mu$ is used until the sensor is repaired. Thus,  
   $\mathcal{A}(i,0)=\mathcal{A}(i,0,t)=\{\mu\}$.
 In states $i\in S$, the action space is $\mathcal{A}(i)=\{t :  t\in \mathbb{R}_{\geq 0}\cup\{\infty\}\} $, with an instantaneous jump to the resulting state.

The incurred costs are the operating costs of service, the \textit{holding} costs for the amount of customers in the system and the costs $c_r$ per sensor repair. 
The service cost rates are $c_{\mu_1},c_{\mu_2}$ for the use of service rates $\mu_1,\mu_2,$ respectively. W.l.o.g., we let $0=c_{\mu_1}<c_{\mu_2}$.  
The holding cost function $h:S\rightarrow \mathbb{R}$ gives that the cost rate for the current number of customers $i$ is $h(i)$, independently of the sensor state.
We impose the following assumption on the holding cost, to impose structural properties and to ensure that the average expected cost using the baseline of service is finite.
\begin{assumption}\label{holdingassump}
  W.l.o.g., $h(0)=0$. 
The holding cost function $h$ is non-negative, non-decreasing and convex in the customer amount, and non-constant. 
Furthermore, 
  $$  \sum_{i=0}^\infty h(i)\cdot x^i <\infty, \ \ \ \ \text{ for all } 0<x<1.$$
\end{assumption} 

 We consider the
stationary deterministic policy space $\Phi$ for our model.
As such, policies $\phi \in \Phi$, are functions that map all states $s\in S$ to an action in their action space.
The set of all stationary deterministic policies is given by policy space $\Phi$. Policy $\phi\in \Phi$ is determined by its actions in states $i,(i,1)\in S$. In states $(i,1)\in S$, policy $\phi$ can be decomposed into two functions $\phi_c:\mathbb{N}_0\rightarrow \{\mu_1,\mu_2\}$ and 
$\phi_t:\mathbb{N}_0\rightarrow \mathbb{R}_{\geq 0}\cup \{\infty\}$, so that $\phi(i,1)=(\phi_c(i),\phi_t(i))$. 
Note that  $\phi$ is completely determined by the pair $(\phi_c,\phi_t)$. For notational convenience we write $\phi=(\phi_c,\phi_t)$.
 The corresponding cost rate function $c^\phi$ is given by $c^\phi(i,0,t)=c^\phi(i,0)=c_\mu+h(i)$ and $c^\phi(i,1)=c_{\phi_c(i)}+h(i)$ for $\phi=(\phi_c,\phi_t)\in \Phi$. 

 In Kanavetas et al.~\cite{Koopmans}, an M/M/1 queue with a single period of service rate control is compared to a baseline M/M/1 queue operating under service rate $\mu$. The total saved costs w.r.t. the baseline model is approximated, and corresponding optimal or approximately optimal policies can be determined explicitly. 
 In essence, this paper compares the cost of the baseline policy $\phi_0=(\mu,\phi^\infty_t)\in \Phi$, where $\phi_t^\infty \equiv \infty$  to the cost of policies of the restricted form $\phi^\infty=(\phi_c,\phi_t^\infty)\in \Phi$. 
 Under this restriction,
  the resulting model is a Markov Decision Process (MDP), which we call the no-repair model in this paper.
In this paper, as opposed to the no-repair model, the option to initiate a service rate control period can be taken repetitively and unlimited. Correspondingly, this paper studies a structural reduction in average expected costs compared to using the baseline policy $\phi_0$ instead of a single cost save.

 In Kanavetas et al.~\cite{Koopmans}, the optimal control policy $\tilde{\phi}_c$ of the no-repair model is determined (generally by applying Lemma~5.1 and validating optimality with Remark~5.1.) and it has a simple threshold form. We denote the threshold of $\Tilde{\phi}_c$ by $i^\ast\in \mathbb{N}_0$.
 
 The dynamics between $\phi_c$ and $\phi_t$ in the general model induces a very complicated optimisation. Hence, we will fix $\phi_c$ to be the optimal no-repair policy $\tilde{\phi}_c$ and proceed to construct policies of the form $\tilde{\phi}=\phi_{A,B,\ell}\in \Phi,$
 where 
 $$\phi_{A,B,\ell}(i,1)=\begin{cases}
     (\tilde{\phi}_c(i),A)& \text{for } i\leq \ell,\\
     (\tilde{\phi}_c(i),B)& \text{for }i>\ell.
 \end{cases} $$ 
 This type of policy with two thresholds is easy to implement in applications and greatly reduces the difficulty of analysing the average expected costs compared to a general choice of $\phi_t$ as a function $S\rightarrow\mathbb{R}_{\geq 0}\cup\{\infty\}$.

 More specifically, in the case $\mu=\mu_1$, we consider policies of the form $\phi_{A,0,\ell}$ and when $\mu=\mu_2$, policies of the form $\phi_{0,B,\ell}$. Intuitively speaking, we consider direct repairs in the region of the state space where service control gives the highest cost reduction (cf. Propositions~\ref{MonotDiff} and \ref{ImprovSet}) and in the other region repair is postponed, so that the number of customers at the sensor repair time is sufficiently close to stationarity of the M/M/1 queue with rate $\mu$. 

Let the system commence in state $s\in S$ under policy $\phi\in \Phi$ and denote the state of the system at time $t$ as $X(t)$. 
 The average expected cost of the process $\{X(t)\}_{t\geq 0}$ is defined as $$g_s^\phi = \limsup_{T\rightarrow \infty} \frac{1}{T} \Bigg[\mathbb{E}_{s}\Bigg[ \int_0^T c^\phi(X(t)) dt\Bigg] + c_r\cdot \mathbb{E}^\phi_{s}\Big[ N(T) \Big] \Bigg],$$
 where $N(t)$ denotes the number of repairs up to time $t$.

 Note that $\{(i,0) :\ i\in \mathbb{N}_0\}\subseteq S$ is a positive recurrent class under any policy with average expected cost
\begin{equation} \label{eq:gf}
    g_\mu= c_{\mu}+\sum_{i=0}^\infty h(i)\cdot \pi_\mu(i),\end{equation}
where $\pi_\mu= (1-\lambda/\mu)(\lambda/\mu)^i$ is the stationary distribution of the baseline M/M/1 queue.

For the candidate policies $\phi_{A,B,\ell}$ with $A,B<\infty$, 
 the set $S\setminus \{(i,0) :\ i\in \mathbb{N}_0\}$ is a closed class containing a single positive recurrent class. $g_s^{\phi_{A,B,\ell}}$ is constant on this positive recurrent class and equal to $g_{(0,1)}^{\phi_{A,B,\ell}}$ which can differ from $g_{\mu}$.





In this paper, we aim to determine average expected cost reducing policies $\phi_{A,B,\ell}$ such that $g^{\phi_{A,B,\ell}}_{(0,1)}<g_\mu$ and we study the size of the resulting average expected cost reduction. The sufficient condition for this methodology to succeed is
that a single period of control gives a cost save that is higher than $c_r$.

The outline of this paper is as follows. 
In Section~\ref{sect:norepair}, the no-repair model of Kanavetas et al.~\cite{Koopmans} is studied.  We determine a value $\ell$ for candidate policies $\phi_{A,B,\ell}$ through the value of the single saved cost due to control in this model. 


In Section~\ref{sect:improv}, sufficient choices of $A,B$ are then derived such that $\phi_{A,B,\ell}$ gives an average expected cost reduction. This section is split into the case $\mu=\mu_1$ in Section~\ref{sect:mu1} and the case $\mu=\mu_2$ in Section~\ref{sect:mu2}. 
Using a vanishing discount, sufficient choices of $A,B$ are derived to guarantee single cost reductions when only repairing the sensor once. 
In both cases, subsequently an average expected cost reduction is derived by rewriting the discounted value function through a telescoping series over the number of sensor repairs. In the case $\mu=\mu_1$, this gives a direct average expected cost reduction, whereas in the case $\mu=\mu_2$, additional restrictions apply in determining the value of $B$ related to a bound on the probability of large queue lengths.

Finally, in Section~\ref{sect:comp},
for several parameter choices the resulting cost reducing policies are calculated with corresponding lower bounds on the average expected cost reductions.
The actual cost reductions of the improving policies are approximated through by simulation 
and are compared to the seemingly pessimistic lower bounds.



\section{The no-repair model}\label{sect:norepair}

In this section, we will analyse the no-repair model, which will provide 
the derivation of a sufficient value of $\ell$ for average expected cost reducing policies $\phi_{A,B,\ell}$ in Section~\ref{sect:ell}, and in the analysis of cost differences. 

In order to use and extend on earlier results by the authors~\cite{Koopmans}, we normalise the parameters (scaling the time and costs) such that $\lambda +\mu_1+\mu_2+\beta=1$. This timescale was needed in \cite{Koopmans} to use uniformisation~\cite{lippman,serfozo} which allows us to study the no-repair model through its discrete time equivalent. 
 
We first formalise the no-repair model.  
In particular, we restrict the state space to the closed class $S^\infty=\{ (i,1),(i,0) :  i\in \N_0 \}$ under no-repair policies.  Consequently, the associated action spaces $\mathcal{A}^\infty$ are defined by  
$\mathcal{A}^\infty(i,1)=\{\mu_1,\mu_2\}$ and $\mathcal{A}^\infty(i,0)=\{\mu\}$ for any $(i,1),(i,0)\in S^\infty$. For notational convenience, also the policy space can be restricted to control policies $\phi_c\in \Phi^\infty$.





The corresponding non-zero transition probabilities
are 

$$ p^{\phi_c}_{(i,1),s} = \begin{cases} \lambda, \   \  
&s=(i+1,1),\\
\beta,  \ \   &s=(i,0),\\
\phi_c(i) \mathds{1}_{\{i>0\}},   \ \ &s=(i-1,1),\\
1-\lambda-\beta-\phi_c(i)\mathds{1}_{\{i>0\}}, \   \ &  s=(i,1),\end{cases} $$

$$ 
   p^{\phi_c}_{(i,0),s}= \begin{cases} \lambda, \   \ &s=(i+1,0),\\
\mu\mathds{1}_{\{i>0\}}, \   \ &s=(i-1,0),\\
1-\lambda-\mu\mathds{1}_{\{i>0\}}, \   \ &s=(i,0).\end{cases}
 $$ 

As cost function, we have the earlier defined cost function $c^{\phi_c}$ and there are no repairs. 

As discount factor we take $1-\alpha$ with $0\leq \alpha <1$, 
 which is the equivalent to the continuous time discount factor $\alpha/(1-\alpha)$. 
The
 expected total discounted cost of the discrete time MDP $\{X_n
 \}_{n\in \mathbb{N
}_0}$ under $\phi_c$, starting in state $s\in S^\infty$ is given by 
\begin{equation}\label{eq:Valphadisc}
    V^{\phi_c}_{\alpha}(s)= \mathbb{E}_{s}^{\phi_c} \Big[ \sum_{n=0}^\infty (1-\alpha)^n c
    (X_n
    ) \Big].
\end{equation}


The discrete time average cost optimality equation (DAOE) for this MDP
is given by

\begin{equation*}\label{CAOE}
    g + H(x)  =\min_{\phi \in \Phi} \left\{ c^\phi(x) + \sum_{y\in S}p^\phi_{x,y}H(t) \right\}, \ \ \ x \in S,
\end{equation*}
 with average expected cost $g\in \mathbb{R}_{+}$, and relative value $ H:S\rightarrow \mathbb{R}$ as solution (cf. \cite{Koopmans}).
I.e., the DAOE for the no-repair model is given by 

\begin{align}\label{modelDAOE}
&g_\infty+H_\infty(0,0) &=& \ c_\mu+(1-\lambda)H_\infty(0,0)+\lambda H_\infty(1,0), \nonumber \\
&g_\infty+ H_\infty(i,0) &=& \ c_\mu+h(i)+\mu H(i-1,0) \nonumber\\
&&& \ \ \ +(1-\mu-\lambda) H_\infty(i,0)+\lambda H_\infty(i+1,0),  \\
&g_\infty +H_\infty(0,1) &=& \ \beta H_\infty(0,0)+(1-\beta-\lambda) H_\infty(0,1)+\lambda H_\infty(1,1), \nonumber \\
&g_\infty +H_\infty(i,1) &=& \  \min_{a\in \{\mu_1,\mu_2\}} \Big\{ c_a+h(i)+\beta H_\infty(i,0)+a H_\infty(i-1,1) \nonumber \\
&&& \hspace{2.22cm}+(1-a-\beta-\lambda) H(i,1)+\lambda H(i+1,1)\Big\}, \nonumber
\end{align}
\noindent  where $i \in \mathbb{N}_{\geq 1}$, $g_\infty=g_\mu$ is the average expected cost (Equation~(\ref{eq:gf})) and $H_\infty$ is the relative value vector.

This model can be studied through the Value Iteration (VI) algorithm, confer \cite{Koopmans}.  With the convention that $i^+=\max\{i,0\}$ for $i\in \mathbb{Z}$, the model-specific algorithm is as follows.

\begin{algorithm}[H]
\begin{algorithmic}[1]
\STATE Set $V_{0}\equiv 0$ and $n=0$.
\STATE For each $(i,1)\in S^\infty$, 
\begin{align*}
     V_{n+1}(i,1)  =  h(i)+ \min_{a\in \{\mu_1,\mu_2\}} \bigg\{ c_a + a V_{n}((i-1)^+,1)&+(1-a- \lambda-\beta) V_{n}(i,1)\bigg\}\\
     & +\lambda V_{n}(i+1,1)+\beta V_{n}(i,0) ,\\
     \phi_{n+1}(i,1)=\argmin_{a\in \{\mu_1,\mu_2\}} \bigg\{ c_a + a V_{n}((i-1)^+,1)&+(1-a- \lambda-\beta) V_{n}(i,1)\bigg\}.
\end{align*}
For $(i,0)\in S^\infty$, 
\begin{align*}
     V_{n+1}(i,0) &=  c_\mu+ h(i)+\mu V_{n}((i-1)^+,0)+(1-\mu- \lambda) V_{n}(i,0)
     +\lambda V_{n}(i+1,0) .
\end{align*}
\STATE Increment $n$ by 1 and return to step 2.
\end{algorithmic}
\caption{Value Iteration for the no-repair model 
}
\label{VIalpha}
\end{algorithm}

 In Kanavetas et al.~\cite{Koopmans} strong Blackwell optimality of the no-repair model has been derived, and convergence of the saved costs as the discount vanishes has been established. This can be formalised as follows.

\begin{restatable}{theorem}{convergenceVIstatementtwo}\label{convergenceVI2}
    There exists a unique solution pair $(g_\infty^\ast,H_\infty^\ast)$ to optimality equation~(\ref{modelDAOE}), with $H_\infty^\ast(0,0)=0$. This value $g_\infty^\ast$ is the average expected cost (independent of the starting state).
    For any $s\in S^\infty$ the following limits exist and hold:
\begin{align*}
        H_\infty^\ast(s) &= \lim_{\alpha\downarrow 0} \big( V^\ast_{\alpha}(s)-V^\ast_{\alpha}(0,0) \big)=\lim_{n\rightarrow \infty}(V_{n}(s)-V_n(0,0)),\\
        g^\ast &= \lim_{\alpha\downarrow 0} \alpha V^\ast_{\alpha}(s),\\
        \phi^\ast &= \lim_{\alpha\downarrow 0} \phi^{\alpha},
    \end{align*}
    with $\phi^\alpha$ being an $\alpha$-discounted optimal control policy 
    and $\phi^\ast$ a strongly Blackwell optimal control policy.
\end{restatable}
The stated convergence towards the average expected cost follows from a Tauberian theorem (see Theorem~\ref{Thm:Tauber}).

Theorem~\ref{convergenceVI2} allows us to derive  characteristics of solutions of the DAOE through VI. These characteristics are needed to derive a sufficient choice of $\ell$ for our construction of $\phi_{A,B,\ell}$. Furthermore, the limits through a vanishing discount are of use for analysis of the general continuous model.

The structural properties of the value vectors are given
in the following proposition from \cite{Koopmans}. 

\begin{proposition} \label{increasingconvexvalues}
    Let $V_{0}\equiv0$, then
    \begin{enumerate}
        \item $V_{n}(i,0)$ and $V_{n}(i,1)$ are non-decreasing and convex in $i\in \mathbb{N}_0$ for any $n\in \mathbb{N}_0$.
        \item  In any step $n$ of VI, the policy $\phi_{n}\in \Phi^\infty$ of chosen actions is of a threshold form. I.e., there exists a threshold $i_n\in \mathbb{N}_0\cup\{\infty\}$, such that\\ $$\phi_{n}(i,1) =\begin{cases}
        \mu_1 \ \ \ \ \ \text{ if } \ \ i\leq i_{n},\\
         \mu_2 \ \ \ \ \ \text{ if } \ \ i> i_{n},
    \end{cases}$$
    where $i_{n}$ is the largest value of $i$ such that 
    \begin{equation}
    \label{eq:in}
    V_{n}(i,1)-V_{n}((i-1)^+,1) \leq \frac{c_{\mu_2}}{\mu_2-\mu_1}.
    \end{equation}
    \item $H_\infty^\ast(i,0)$ and $H_\infty^\ast(i,1)$ are non-decreasing and convex in $i$ and optimal policies are of a threshold form with threshold $i^\ast\in \mathbb{N}_0$.
    \end{enumerate}
\end{proposition}

The threshold structure reflects increasing marginal congestion costs: as queue length grows, the value of faster service eventually exceeds its additional operating cost, leading to a single switching point.

Proposition~\ref{increasingconvexvalues} allows us to show the following characteristic of the model.
When the baseline service rate equals the slower rate 
$\mu_1$, restoring control becomes increasingly valuable as congestion grows. Conversely, when the baseline service rate equals $\mu_2$, additional control is most valuable at smaller queue lengths.

\begin{proposition} \label{MonotDiff}
 For $\mu=\mu_1$, we have that $0\leq H^\ast_\infty(i,0)-H^\ast_\infty(i,1)$ is non-decreasing in $i\in \mathbb{N}_0$.
For $\mu=\mu_2$, we have that $0\leq H_\infty^\ast(i,0)-H_\infty^\ast(i,1)$ is non-increasing in $i\in \mathbb{N}_0$.
\end{proposition}
\begin{proof} 

  We will use induction on the value vectors of VI. The statement that we first need to prove is that,  for any $n\in \mathbb{N}_0$, we have that $0\leq V_{n}(i,0)-V_{n}(i,1)$ is non-decreasing in $i\in \mathbb{N}_0$ when $\mu=\mu_1$ and non-increasing in $i\in \mathbb{N}_0$ when $\mu=\mu_2$. The base case $n=0$ is true, as we take $V_{0}=0$.

Assume that the above statement is true upto and including iteration $n$, for both $\mu=\mu_1$ and $\mu=\mu_2$.
Let first $\mu=\mu_1$.  For $0\leq  i< i_{n+1}$, 
\begin{eqnarray*}
\lefteqn{V_{n+1}(i,0)-V_{n+1}(i,1)}\\
&=& h(i)+\lambda V_n(i+1,0)+(1-\lambda-\mu_1)V_n(i,0)+\mu_1 V_n((i-1)^+,0)\\
&&\quad -h(i)-\lambda V_n(i+1,1)-(1-\lambda-\mu_1-\beta)V_n(i,1)-\beta V_n(i,0)-\mu_1 V_n((i-1)^+,1)\\
&=&\lambda(V_n(i+1,0)-V_n(i+1,1))+(1-\lambda -\mu_1-\beta)(V_n(i,0)-V_n(i,1))\\
&&\quad + \mu_1V_n((i-1)^+,1)-V_n((i-1)^+,0)\\
&\leq&\lambda(V_n(i+2,0)-V_n(i+2,1))+(1-\lambda -\mu_1-\beta)(V_n(i+1,0)-V_n(i+1,1))\\
&&\quad +\mu_1(V_n(i),1)-V_n(i,0))=V_{n+1}(i+1,0)-V_{n+1}(i+1,1).
\end{eqnarray*}
The last inequality follows from the induction hypothesis.
The last equality follows by replacing $i$ in the first two equalities by $i+1$, 
and the fact that the minisiming action in state $(i+1,1)$   is equal to the one in $(i,1)$.

Next, consider the case $i> i_{n+1}$. This yields
\begin{eqnarray*}
\lefteqn{V_{n+1}(i,0)-V_{n+1}(i,1)}\\
&=& h(i)+\lambda V_n(i+1,0)+(1-\lambda-\mu_1)V_n(i,0)+\mu_1 V_n((i-1)^+,0)\\
&&\quad -c_{\mu_2}-h(i)-\lambda V_n(i+1,1)-(1-\lambda-\mu_2-\beta)V_n(i,1)-\beta V_n(i,0)-\mu_2 V_n((i-1)^+,1)\\
&=&-c_{\mu_2}+\lambda(V_n(i+1,0)-V_n(i+1,1))+(1-\lambda -\mu_1-\beta)(V_n(i,0)-V_n(i,1))\\
&&\quad + \mu_1(V_n((i-1)^+,0)-V_n((i-1)^+,1))+(\mu_2-\mu_1)( V_n(i,1)- V_n((i-1)^+,1))\\
&\leq&-c_{\mu_2}+\lambda(V_n(i+2,0)-V_n(i+2,1))+(1-\lambda -\mu_1-\beta)(V_n(i+1,0)-V_n(i+1,1))\\
&&\quad +\mu_1(V_n((i,0)-V_n(i,1))+(\mu_2-\mu_1)( V_n(i+1,1)- V_n(i,1))\\
&=&V_{n+1}(i+1,0)-V_{n+1}(i+1,1).
\end{eqnarray*}
The last inequality follows from the induction hypothesis,  convexity of $V_n$ in $i$, and $\mu_2\geq\mu_1$.

Finally, we consider the case $i=i_{n+1}$. Then,
\begin{eqnarray*}
\lefteqn{V_{n+1}(i,0)-V_{n+1}(i,1)}\\
&=& h(i)+\lambda V_n(i+1,0)+(1-\lambda-\mu_1)V_n(i,0)+\mu_1 V_n((i-1)^+,0)\\
&&\quad -h(i)-\lambda V_n(i+1,1)-(1-\lambda-\mu_1-\beta)V_n(i,1)-\beta V_n(i,0)-\mu_2 V_n((i-1)^+,1)\\
&=&\lambda(V_n(i+1,0)-V_n(i+1,1))+(1-\lambda -\mu_1-\beta)(V_n(i,0)-V_n(i,1))\\
&&\quad + \mu_1(V_n((i-1)^+,0)-V_n((i-1)^+,1))\\
&\leq& \lambda(V_n(i+2,0)-V_n(i+2,1))+(1-\lambda -\mu_1-\beta)(V_n(i+1,0)-V_n(i+1,1))\\
&&\quad + \mu_1(V_n(i,0)-V_n(i,1))\\
&=&V_{n+1}(i+1,0)- V_{n+1}(i+1,1) +
c_{\mu_2}+ (\mu_1-\mu_2) (V_n(i+1,1)-V_n(i,1))\\
&\leq &V_{n+1}(i+1,0)- V_{n+1}(i+1,1),
\end{eqnarray*}
where in the one but last inequality we have used the induction hypothesis. In the last equality, we have used the expressions for  $V_{n+1}(i+1,0)$ and $V_{n+1}(i+1,1)$.
Finally, in the last inequality, we have used that $c_{\mu_2}\leq (\mu_2-\mu_1) (V_n(i,1)-V_n((i-1)^+,1))$, cf. Proposition~\ref{increasingconvexvalues}, and  Eqn.(~\ref{eq:in}).

Next, let $\mu=\mu_2$. For $i> i_{n+1}$, in  $(i,0)$, $(i+1,0)$,  $(i,1)$ and $(i+1,1)$  service rate $\mu_2$ is used. Therefore, the
induction step is analogously validated as in  the case $\mu=\mu_1$ and $i<i_{n+1}$. Similarly, for $i<i_{n+1}$, the service rates in $(i,1)$ and $(i+1,1)$ are equal  to $\mu_1$, and so the induction step can be validated similarly to the case $i>i_{n+1}$ for $\mu=\mu_1$.

We check the case that $i=i_{n+1}$, for which the service rate in $(i,1)$ is $\mu_1$ and $\mu_2$ in $(i+1,1)$.
\begin{eqnarray*}
\lefteqn{V_{n+1}(i,0)-V_{n+1}(i,1)}\\
&=& c_{\mu_2}+h(i)+\lambda V_n(i+1,0)+(1-\lambda-\mu_2)V_n(i,0)+\mu_2 V_n((i-1)^+,0)\\
&&\quad -h(i)-\lambda V_n(i+1,1)-(1-\lambda-\mu_1-\beta)V_n(i,1)-\beta V_n(i,0)-\mu_1 V_n((i-1)^+,1)\\
&=&c_{\mu_2}+\lambda(V_n(i+1,0)-V_n(i+1,1))+(1-\lambda -\mu_2-\beta)(V_n(i,0)-V_n(i,1))\\
&&\quad +\mu_2(V_n((i-1)^+,0)-V_n((i-1)^+,1) +( \mu_1-\mu_2)(V_n(i,1)-V_n((i-1)^+,1)))\\
&\geq&c_{\mu_2}+ \lambda(V_n(i+2,0)-V_n(i+2,1))+(1-\lambda -\mu_2-\beta)(V_n(i+1,0)-V_n(i+1,1))\\
&&\quad +\mu_2(V_n(i,0)-V_n(i,1))+ (\mu_1-\mu_2)(V_n(i,1))-V_n((i-1)^+,1))\\
&=&V_{n+1}(i+1,0)- V_{n+1}(i+1,1) +
c_{\mu_2}+ (\mu_1-\mu_2) (V_n(i+1,1)-V_n(i,1))\\
&\geq &V_{n+1}(i+1,0)- V_{n+1}(i+1,1),
\end{eqnarray*}
where in the  one but last inequality we have again used the induction hypothesis.
In the last equality, we have used the expressions for  $V_{n+1}(i+1,0)$ and $V_{n+1}(i+1,1)$.
Finally,  (minus) Eqn.(~\ref{eq:in}) is used in the last inequality.

     Finally, by virtue of  Theorem~\ref{convergenceVI2},  
     $$H_\infty^\ast(i,0)-H_\infty^\ast(i,1)=\lim_{n\rightarrow \infty}V_{n}(i,0)-V_{n}(i,1),
     $$ and thus these differences are non-decreasing in $i$ for $\mu=\mu_1$, and non-increasing for $\mu=\mu_2$.
\end{proof}


\subsection{Determining threshold $\ell$ for policies $\phi_{A,B,\ell}$}\label{sect:ell}

The choice of $\ell$ is based on the critical cost, which is defined as follows.

\begin{definition}\label{CritCost}
    The value $c_r^\ast=\sum_{i=0}^\infty\pi_\mu(i) (H^\ast_\infty(i,0)-H^\ast_\infty(i,1))$ is called the \textit{critical cost} and is the cost saved by a single period of control starting in stationary distribution $\pi_\mu$.
\end{definition}

How to determine/approximate $c^\ast_r$ and the fact that it exists is treated in \cite{Koopmans}.


We choose the threshold $\ell$ such that the cost save on one side of the threshold is at least $c^\ast_r$.
For the general model this has the following reasoning. 
On the corresponding side of $\ell$,  $\phi_{A,B,\ell}$ will prescribe a direct sensor repair. In this region, we want to guarantee a future cost reduction of at least $c^\ast_r$, in the region where this is not guaranteed, we would like to wait untill the distribution of the number of customers is close to $\pi_\mu$, implying that the next cost reduction is close to $c^\ast_r$.
Thus, when $c_r<c^\ast_r$ this methodology will result in an average expected cost reduction.

To determine $\ell$, we use Proposition~\ref{MonotDiff} to find relative value differences that are larger than the critical cost.

\begin{proposition}\label{ImprovSet}
\hspace{0.001cm}
\begin{enumerate}
    \item For $\mu=\mu_1$, there exists an $\ell\in \mathbb{N}_0$ such that $ h(\ell)>g_\infty+\mu_1(c^\ast_r+c_{\mu_2})/(\mu_2-\mu_1) $ and\\ $(\mu_2-\mu_1)(h(\ell)-h(\ell-1)) >c_{\mu_2} $,
    implying that $H_\infty^\ast(i,0)-H_\infty^\ast(i,1)\geq c^\ast_r$, for $i\geq \ell$.
    \item  For $\mu=\mu_2$, and the choice $\ell=0$, it holds that
    $H_\infty^\ast(i,0)-H_\infty^\ast(i,1)\geq c^\ast_r$, for $i\leq \ell$.
\end{enumerate}

\end{proposition}
\begin{proof}
 The existence of such a value $\ell$ follows directly by Proposition~\ref{MonotDiff} and the fact that the critical cost is a weighted average of $H_\infty^\ast(i,0)-H_\infty^\ast(i,1)$ w.r.t. distribution $\pi_\mu$. The sufficient choice of $\ell=0$ when $\mu=\mu_2$ follows directly.

    We now consider the case $\mu=\mu_1$. Let any $\ell$ be given such that the conditions hold. We note that such an $\ell$ exists due to Assumption~\ref{holdingassump}.
    
    By Proposition~3.2 of \cite{Koopmans} 
    $(\mu_2-\mu_1) (V_n(i,1)-V_n(i-1,1))\geq (\mu_2-\mu_1)( h(\ell)-h(\ell-1))>c_{\mu_2}$, for $i\geq \ell$, and therefore $i^\ast\leq i_{1} \leq \ell-1$ .
    We will now show that $H_\infty^\ast(\ell,0)-H_\infty^\ast(\ell,1)\geq c^\ast_r. $ 
We refer to the DAOE given in Equation~(\ref{modelDAOE}).
We also use the fact that $H_\infty^\ast(i,0)-H_\infty^\ast(i,1)\geq 0$ for any $i\in \mathbb{N}_0$.
We find
\begin{eqnarray}
\label{eq:kprop}
H_\infty^\ast(\ell,0)-H_\infty^\ast(\ell,1) &=& -c_{\mu_2} + \mu_1( H_\infty^\ast(\ell-1,0)-H_\infty^\ast(\ell,1))  
     +\mu_2 ( H_\infty^\ast(\ell,0)-H_\infty^\ast(\ell-1,1)) \nonumber\\
&&\quad+\beta\cdot 0 +\lambda( H_\infty^\ast(\ell+1,0)-H_\infty^\ast(\ell+1,1))\nonumber\\
&\geq& -c_{\mu_2} + \mu_1( H_\infty^\ast(\ell-1,0)-H_\infty^\ast(\ell,1))  +\mu_2 ( H_\infty^\ast(\ell,0)-H_\infty^\ast(\ell-1,1))\nonumber\\
&= &-c_{\mu_2}+\mu_1 ( H_\infty^\ast(\ell,0)-H_\infty^\ast(\ell,1))+\mu_2(H_\infty^\ast(\ell-1,0)-H_\infty^\ast(\ell-1,1))\nonumber \\
&&\quad  +(\mu_2-\mu_1)(H_\infty^\ast(\ell,0)-H_\infty^\ast(\ell-1,0))\nonumber\\
 & \geq& -c_{\mu_2}+(\mu_2-\mu_1)(H_\infty^\ast(\ell,0)-H_\infty^\ast(\ell-1,0)).
\end{eqnarray}

Then, by the DAOE of Equation~(\ref{modelDAOE}) we can derive that \begin{equation}\label{Hiter} g_\infty-h(i)+\mu_1(H_\infty(i,0)-H_\infty(i-1,0))=\lambda(H_\infty(i+1,0)-H_\infty(i,0)) . \end{equation} 
Equation~(\ref{Hiter}) and the non-decreasingness of $H_\infty(i,0)$ in $i$ given by Proposition~\ref{increasingconvexvalues} result in 
\begin{equation*}
    \begin{split}
        & g_\infty-h(i)+\mu_1(H_\infty(i,0)-H_\infty(i-1,0))\geq 0\\
        &\quad \implies  \mu_1(H_\infty(i,0)-H_\infty(i-1,0))\geq h(i)-g_\infty.
    \end{split}
\end{equation*}
With our choice of $\ell$ such that $ h(\ell)>g_\infty+\mu_1(c^\ast_r+c_{\mu_2})/(\mu_2-\mu_1) $ this implies that \\
$ H_\infty(\ell,0)-H_\infty(\ell-1,0)\geq (c^\ast_r+c_{\mu_2})/(\mu_2-\mu_1)$. Together with Equation~(\ref{eq:kprop}), we can conclude that $ H_\infty(\ell,0)-H_\infty(\ell,1)\geq c^\ast_r$.

\end{proof}

\noindent In the case $\mu=\mu_1$, we will consider improving policies $\phi_{A,0,\ell}$. As
$i^\ast<\ell$, we note that there is no direct sensor repair in states where service rate $\mu_1$ is already being prescribed. 
In the case $\mu=\mu_2$, trivially $i^\ast>\ell=0$ and a similar statement holds.

\noindent 
In the case that $\mu=\mu_2$ we note that higher values of $\ell$ than 0 can be preferable and Lemma 4.1 of \cite{Koopmans} can be used to verify whether $H_\infty^\ast(\ell,0)-H_\infty^\ast(\ell,1)\geq c_r^\ast$. These values of $\ell$ are exactly the sufficient value of $\ell$ that we take when constructing average expected cost reducing policies $\phi_{A,B,\ell}$.

\section{Repetitive options for control periods}\label{sect:improv}

In this section, we will determine sufficient choices of $A,B$ 
for average expected cost reducing policies $\phi_{A,B,\ell}$.
The used methodology is similar for the two cases $\mu=\mu_1$ and $\mu=\mu_2$.
First, we derive 
a sufficiently high
 value of $A$ for the case $\mu=\mu_1$ such that a single (potentially delayed) sensor repair is beneficial for every queue length.  A high enough value of $A$ ensures a small enough total variation distance uniformly in initial queue lengths $0\leq i\leq \ell$. Hence, these resulting single cost saves are all close to $c^\ast_r-c_r$.

 We note that $\phi_{A,0,\ell}$ 
induces a renewal reward process.
Hence, a Tauberian theorem allows to study the average expected costs through a vanishing discount.
Subsequently, we use a telescoping sequence over the number of sensor repairs under this vanishing discount to determine a lower bound of the resulting average expected cost reduction.

In the case $\mu=\mu_2$, it is not possible to bound the total variation distance to $\pi_\mu$ at repair time $B$ uniformly over all initial queue lengths $i>\ell$. 
To tackle this problem, we 
accept non-beneficial sensor repairs at cost $c_r$ given that these are compensated in the long run by the beneficial sensor repairs.  For this matter, $B$ is chosen to be large enough to guarantee this.

Consider the continuous time model where we have the option to schedule successive repair times. Here, uniformisation is not possible and so we define the total discounted cost starting in state $s$,  denoted by $V^\phi_\alpha(s)$, as
\begin{equation}
    V_\alpha^\phi(s) = \mathbb{E}^\phi_s\Bigg[\int_{t=0}^\infty e^{-\alpha t} c^\phi(X_t)  dt+\sum_{k=1}^{\infty} e^{-\alpha R(k)}\cdot c_r \Bigg], \label{eq:Valpha}
\end{equation}
where $R(k)$ is the time of the $k$-th repair. Note that, under transformation of the discount factor, Equations~(\ref{eq:Valphadisc}) and (\ref{eq:Valpha}) for no-repair policies agree and as one discount vanishes, so does the other.



In order to quantify the cost reductions of size $c^\ast_r-c_r$, we introduce the critical gap and assume that it is positive.  

\begin{definition}\label{DriftDef}
    The \textit{critical gap} $\Delta_c$ is defined to be $\Delta_c:=c^\ast_r-c_r$.
\end{definition}


\begin{assumption}\label{criticalassump} $c_r<c^\ast_r$ such that the critical gap is positive, i.e., $\Delta_c>0$.
\end{assumption} 
\noindent
As a result of Theorem~6.2 of \cite{Koopmans}, we can  derive a positive lower bound on the critical gap.

 
We will show that Assumption~\ref{criticalassump} is sufficient for the derivation of average expected cost reducing policies $\phi_{A,B,\ell}$. The constructed average expected cost reducing policies and lower bounds of their average expected cost reductions are provided in Theorems~\ref{mufismu1avg} and \ref{mufismu2avg} for the cases $\mu=\mu_1$ and $\mu=\mu_2$ respectively.




To measure the  total variation distance, we consider the  probabilities $P_{ij}(t)$ of being in state $j$ at time $t$, when starting in $i$,  using service rate $\mu$. The values $P_{ij}(t)$ converge to $\pi_\mu(j)$, as $t\rightarrow \infty$. 
Lund et al.~\cite{LundTweedie} derived a bound on the total variation distances $\sum_{j=0}^\infty|P_{ij}(t)-\pi_\mu(j)|$ for stochastically ordered Markov processes. We use 
 a version of the specific bounds for the M/M/1 queue given by Robert~\cite{Robert2003}. The bound quantifies how quickly the uncontrolled queue forgets its initial state. This determines how long one should wait before restoring control.

\begin{proposition} \label{TVexpconverg}
 For any starting state $i\in \mathbb{N}_0$ and time $t\geq 0$, the distance to the equilibrium satisfies 
 $$\sum_{j=0}^\infty | P_{ij}(t)-\pi_\mu(j) | \leq 2(1+
 \rho_\mu^{-i/2})e^{-(\sqrt{\mu}-\sqrt{\lambda})^2t} . $$
 \end{proposition}

\subsection{Fixed rate $\mu=\mu_1$}\label{sect:mu1}

When $\mu=\mu_1$, Proposition~\ref{ImprovSet} gives a sufficient value of $\ell$ with the incentive for direct repairs when the number of customers $i$ satisfies $i>\ell$. Hence, we will consider candidates for average expected cost reductions of the form $\phi_{T,0,\ell}$.
In order to analyse these candidate policies, we will, as a matter of exception consider a sequence of history dependent policies.

For any $T\geq 0$, we define $\phi_{T,0,\ell}^{j}$ as the (history dependent) policy that  pays $c_r$ for a direct repair when control is lost in states $(i,1)$ with $i> \ell$, and for all other states, first waits $T$ time before paying $c_r$ for a repair. After exactly $j\in \mathbb{N}$ repairs, the sensor will never be repaired again. During control, we take the control actions as prescribed by control policy $\Tilde{\phi}_c$.
We will show that, under Assumption~\ref{criticalassump}, there exists a time $T\geq 0$, such that $\phi_{T,0,\ell}^{1}$ gives a single cost reduction over not paying for a control period at all, for all starting states.  subsequently, we will show that in this case, we can reduce the average expected costs by a telescoping sum over the sequence $\{\phi_{T,0,\ell}^{j}\}_{j\in \mathbb{N}}$.
The waiting time $T$  balances two effects: waiting longer improves proximity to equilibrium, but delays the benefit of restoring control.

In order to determine a sufficient choice of $T$, we need to use an upper bound of $H^\ast_{\infty}(k,0)-H^\ast_{\infty}(k,1)$ for a later specified value of $k\in \mathbb{N}$. As explained in Kanavetas et al.~\cite{Koopmans}, the value $H^\ast_{\infty}(k,0)$ can be determined explicitly and  for the $\epsilon$-approximation $\tilde{H}_\epsilon(k,1)$, determined by Lemma~5.1 of \cite{Koopmans}, it holds that $|\tilde{H}_\epsilon(k,1)-\tilde{H}_\epsilon(k,1)| \leq \epsilon $. As such we can explicitly determine 
\begin{equation}\label{eq:Ukdef}
    U_k= H^\ast_{\infty}(k,0)-\tilde{H}_1(k,1)+1
\end{equation}
as an upper bound of $H^\ast_{\infty}(k,0)-H^\ast_{\infty}(k,1)$ for any $k\in \mathbb{N}$. We note that this upper bound is also valid in the case $\mu=\mu_2$.

\begin{restatable}{theorem}{TMu1}\label{mufismu1}
Let \begin{equation}\label{eq:T1def}
    T = \frac{\log(6 U_k(1+\rho_\mu^{-\ell/2})/\Delta_c)}{(\sqrt{\mu}-\sqrt{\lambda})^2},
\end{equation}
with $k=\tilde{k}_{2\Delta_c/3}$  from Equation~(6.2) in \cite{Koopmans} and with  $U_k$  defined in Equation~(\ref{eq:Ukdef}). 
   Then,
     \begin{equation*}
    \lim_{\alpha \downarrow 0} \Big(   V^{\phi_0}_{\alpha}(i) -  V^{\phi^1_{T,0,\ell}}_{\alpha}(i)\Big)\geq \frac{\Delta_c}{3}, \qquad
    \text{ for } 
    i\in S.
   \end{equation*}
\end{restatable}
\begin{proof}
We 
will use Theorem~\ref{convergenceVI2}.
    For states $i> \ell$, we get 
    \begin{align*}
    \lim_{\alpha \downarrow 0} \Big(   V^{\phi_0}_{\alpha}(i) -  V^{\phi^1_{T,0,\ell}}_{\alpha}(i)\Big) &=  \lim_{\alpha \downarrow 0} \Big(   V^{\ast}_{\alpha}(i,0) -  V^\ast_{\alpha}(i,1)\Big) -c_r  \\
    &= H^\ast_\infty(i,0)-H^\ast_\infty(i,1)-c_r
    \geq c^\ast_r-c_r = \Delta_c.
   \end{align*}
   By Theorem~6.2 from  \cite{Koopmans} for $k=\tilde{k}_{2\Delta_c/3}$ it holds that $$\sum_{i=0}^k \pi_\mu(i)(H^\ast_\infty(i,0)-H^\ast_\infty(i,1)) \geq c^\ast_r-\frac{\Delta_c}{3}.$$
   For states $i\leq\ell$, we use this value of $k$ and find 
   \begin{align} \label{eq:mufmu1p1}
   \begin{split}
   & \lim_{\alpha \downarrow 0} \Big(   V^{\phi_0}_{\alpha}(i) -  V^{\phi^1_{T,0,\ell}}_{\alpha}(i)\Big) =  \lim_{\alpha \downarrow 0} e^{-\alpha T}\sum_{j=0}^\infty P_{ij}(T)  \Big(V^{\phi_0}_{\alpha}(j,0) -  V^\ast_{\alpha}(j,1)\Big) -c_r \cdot e^{-\alpha \cdot T} \\
  & \quad \quad \geq \lim_{\alpha \downarrow 0} \sum_{j=0}^k P_{ij}(T)  \Big(V^{\phi_0}_{\alpha}(j,0) -  V^\ast_{\alpha}(j,1)\Big) -c_r = \sum_{j=0}^k P_{ij}(T)  \Big(H^\ast_\infty(j,0) -  H^\ast_\infty(j,1)\Big) -c_r \\
  &\quad \quad \geq \sum_{j=0}^k \pi_\mu(j)  \Big(H^\ast_\infty(j,0) -  H^\ast_\infty(j,1)\Big)- \sum_{j=0}^k \Big |P_{ij}(T)-\pi_\mu(j)\Big |  \Big(H^\ast_\infty(j,0) -  H^\ast_\infty(j,1)\Big) -c_r \\
  &\quad \quad\geq c^\ast_r- \frac{\Delta_c}{3} -c_r- \sum_{j=0}^k \Big |P_{ij}(T)-\pi_\mu(j)\Big |  \Big(H^\ast_\infty(j,0) -  H^\ast_\infty(j,1)\Big) \\
 & \quad \quad =  \frac{2\Delta_c}{3} - \sum_{j=0}^k \Big |P_{ij}(T)-\pi_\mu(j)\Big |  \Big(H^\ast_\infty(j,0) -  H^\ast_\infty(j,1)\Big).
     \end{split}
   \end{align}
   Next, we bound the sum $\sum_{j=0}^k \Big |P_{ij}(T)-\pi_\mu(j)\Big |  \Big(H^\ast_\infty(j,0) -  H^\ast_\infty(j,1)\Big)$.   Proposition~\ref{MonotDiff} gives that\\
   $0\leq H^\ast_\infty(i,0) -  H^\ast_\infty(i,1)\leq H^\ast_\infty(k,0) -  H^\ast_\infty(k,1)$, for $i\leq k$. Using that $H^\ast_\infty(k,0) -  H^\ast_\infty(k,1) \leq U_k$ and plugging in Proposition~\ref{TVexpconverg} and our choice of $T$ results in
   \begin{equation}
   \begin{split}
       \label{eq:mufmu1p2}
       &\sum_{j=0}^k \Big |P_{ij}(T)-\pi_\mu(j)\Big |  \Big(H^\ast_\infty(j,0) -  H^\ast_\infty(j,1)\Big)
       \leq U_k \sum_{j=0}^k \Big |P_{ij}(T)-\pi_\mu(j)\Big |  \\
      &\quad\quad \quad  \leq U_k \sum_{j=0}^\infty \Big |P_{ij}(T)-\pi_\mu(j)\Big | \leq 2 U_k(1+\rho_\mu^{-i/2})e^{-(\sqrt{\mu}-\sqrt{\lambda})^2 T}=\frac{\Delta_c}{3}.
        \end{split}
   \end{equation}
   Combining Equations~(\ref{eq:mufmu1p1}) and (\ref{eq:mufmu1p2}) then gives
   \begin{align*} 
       \lim_{\alpha \downarrow 0} \Big(   V^{\phi_0}_{\alpha}(i) -  V^{\phi^1_{T,0,\ell}}_{\alpha}(i)\Big)\geq   \frac{2\Delta_c}{3}-\frac{\Delta_c}{3}=  \frac{\Delta_c}{3}.
   \end{align*}

\end{proof}
\noindent In fact, we note that we can determine sufficiently large values of $T$ such that\\
$  \lim_{\alpha \downarrow 0} \Big(   V^{\phi_0}_{\alpha}(i) -  V^{\phi^1_{T,0,\ell}}_{\alpha}(i)\Big)\geq \epsilon$, for any $0<\epsilon< \Delta_c$. 

Furthermore,  the process under policy $\phi_{T,0,\ell}$ regenerates at successive returns to state $(0,1)$, thus inducing a renewal structure.
\begin{restatable}{proposition}{RenRew}\label{renrewmu1}
For any value of  $ T\in \mathbb{R}$, the (non-discounted) process under policy $\phi_{T,0,\ell}$ induces a renewal reward process defined  by random variables $S_i,W_i$ with $i\in \mathbb{N}_{\geq 1}$, that denote the time and costs between the $(i-1)$-th and $i$-th entrances in state $(0,1)$, respectively. 
For this renewal reward process both $\mathbb{E}[S_1]<\infty$ and $\mathbb{E}[|W_1|]<\infty$.
\end{restatable}
\begin{proof}
    As the policy is Markovian, both the renewal interval lengths and the costs per renewal are i.i.d.  

We give an upper bound of 
$\mathbb{E}[S_1]$ by analysing the process and its busy periods (that is; time intervals with at least one customer in the system).
    Starting in state $(0,1)$, the first jump occurs after an expected time of $1/(\lambda+\beta)$. It is either an arrival and the first busy period commences, or it is a sensor breakdown and we jump to state $(0,0,T)$. 
    
    Since $\lambda<\mu=\mu_1$, we note that the expected duration of a busy period is bounded from above by the expected duration of the busy period in the $M/M/1$ queue using only $\mu_1$, which is $1/(\mu_1-\lambda)<\infty$. 
    At the end of a busy period, we are either in $(0,1)$ or in a state $(0,0,t)$ with $0<t\leq T$. 
    
    In states  $(0,0,t)$ with $0<t\leq T$, with a probability of at least $e^{-\lambda T}$, we reach sensor repair and state $(0,1)$ before a customer joins the system. If this does not happen,  the next busy period starts after at most $T$ time units. Hence, this next busy period would end within an expected $T+1/(\mu_1-\lambda)$ time units.
    
    Therefore, the expected time $\mathbb{E}[S_1]$ is bounded from above by $1/(\lambda+\beta)$ plus the expected amount of busy periods before reaching $(0,1)$ multiplied by $T+1/(\mu_1-\lambda)$. It remains to calculate the expected number of busy periods till renewal. 
    
    Recall that the success probability of reaching $(0,1)$ after a busy period before the next busy period commences is at least $e^{-\lambda T}$. Using the expectation of the geometric distribution, the expected amount of busy periods before reaching $(0,1)$ is at most $
    e^{\lambda T} $. 
    
    Thus, we find the upper bound $$\mathbb{E}[S_1] \leq \frac{1}{\lambda+\beta} +e^{\lambda T}\Big(T+\frac{1}{\mu_1-\lambda}\Big). $$
    
    Similarly, the expected service rate costs between renewals are upper bounded by \\
    $$c_{\mu_2}\cdot \mathbb{E}[S_1] \leq c_{\mu_2}\cdot \Big( \frac{1}{\lambda+\beta} +e^{\lambda T}\Big(T+\frac{1}{\mu_1-\lambda}\Big)\Big) .$$

   For the holding costs, we can use a similar argument. The holding costs of one busy period are maximal when only service rate $\mu_1$ is used. This follows, as the amount of customers present in the process when only using $\mu_1$, is stochastically larger or equal to the amount of customers in the controlled processes, at any time $t$. Holding costs are only incurred in the busy period (as $h(0)=0$). Therefore, the holding costs of one busy period  using only $\mu_1$ is an upper bound for the expected holding costs in one busy period. 
   
   The average expected holding costs using only $\mu_1$ is finite by Assumption~\ref{holdingassump} and given by $$g_h:=\sum_{i=0}^\infty \pi_\mu(i)\cdot h(i) <\infty.$$ 
  As mentioned before,  the moments that this system reaches state $(0,1)$ are renewal times. This induces a renewal reward process. The expected time  between successive renewals is the expected time until the first arrival plus the expected time of the busy period, which equals $1/\lambda+1/(\mu_1-\lambda)
   =\mu_1/(\lambda(\mu_1-\lambda))
   $.
   Using the renewal theorem (and the fact that holding costs are only incurred during the busy period), we find expected total holding costs $g_h\cdot 
   \mu_1/(\lambda(\mu_1-\lambda))$ per busy period. Using similar arguments, we find an upper bound of the total holding costs between two consecutive entrances of $(0,1)$ which is given by $$e^{\lambda T}\cdot g_h\cdot \frac{\mu_1}{\lambda(\mu_1-\lambda)}<\infty.$$ 
   Combining service rate cost and holding costs gives 
   $$ \mathbb{E}[|W_1|]\leq   c_{\mu_2} \Big( \frac{1}{\lambda+\beta} +e^{\lambda T}\Big(T+\frac{1}{\mu_1-\lambda}\Big)\Big)+e^{\lambda T}\cdot g_h\cdot \frac{\mu_1}{\lambda(\mu_1-\lambda)}<\infty.$$
   
\end{proof}

To relate the average expected cost of the renewal reward process to a vanishing discount approach we need a Tauberian theorem. Note that the following Tauberian theorem follows from Widder \cite{Widder}  and can be found explicitly in Blok and Spieksma \cite{mdppractice} and in Sennott \cite{sennott}.

\begin{theorem}\label{Thm:Tauber}[Tauberian theorem]
    Let $s$ be a non-negative measurable function such that $S_T:= \int_0^T s(t)dt$ for $T\geq 0$ and $V(\alpha):=\int_0^\infty e^{-\alpha t} s(t) dt$ for $\alpha>0$. Then, \begin{equation*}
       \liminf_{T\rightarrow \infty}\frac{S_T}{T} \leq \liminf_{\alpha\downarrow 0}\alpha V(\alpha) \leq  \limsup_{\alpha\downarrow 0}\alpha V(\alpha)\leq  \limsup_{T\rightarrow \infty}\frac{S_T}{T}.
    \end{equation*}
\end{theorem}

Note that the expected repair costs in $V_\alpha^{\phi_{A,B,\ell}}(i)$ given in Equation~(\ref{eq:Valpha}) for initialisation state $i\in S$  matches this form with the addition of a constant, as we can write
 $$\mathbb{E}_i^{\phi_{A,B,\ell}}\Big[ \sum_{k=1}^\infty e^{-\alpha R(k)} \cdot c_r\Big]=c_r\cdot e^{-\alpha R(1)}+ c_r\cdot \int_{t=R(1)}^\infty e^{-\alpha t} n(t) dt,$$ where $R(1)\in \{A,B\}$ is deterministic and $n(t)$ is the density of $\mathbb{E}^{\phi_{A,B,\ell}}_i[N(t)]$ with respect to the Lebesgue measure.   To show absolute continuity of $\mathbb{E}^{\phi_{A,B,\ell}}_i[N(t)]$ as a function of $t$, it suffices to show Lipschitz continuity. This follows 
from the fact that the losses of control occur with exponential rate $\beta$ so that $$0\leq \mathbb{E}_i^{\phi_{A,B,\ell}}[N(t+\delta t)] - \mathbb{E}_i^{\phi_{A,B,\ell}}[N(t)]\leq 2\beta \delta t$$ for $t\geq R(1)$ and $\delta t>0$ (note that the number of repairs in a period can not be larger than the number of sensor failures in an interval of this length).


Proposition \ref{renrewmu1}, the elementary renewal reward theorem and the Tauberian theorem combined give the following corollary.

\begin{corollary}\label{taubercor}
The process under policy
$\phi_{T,0,\ell} $  has a finite average expected cost $g_{(0,1)}^{\phi_{T,0,\ell} }$, which is equal to $\lim_{\alpha \downarrow 0} \alpha V_\alpha^{\phi_{T,0,\ell}}(i)$.
\end{corollary}

Next, we have to show that we can reduce the average expected costs by repetitive sensor repairs by considering $\phi_{T,0,\ell}$ through the sequence $\{\phi^j_{T,0,\ell}\}_{j\in \mathbb{N}}$.


\begin{restatable}{theorem}{TMu1avg}\label{mufismu1avg}
   Let $T$ be given by Equation~(\ref{eq:T1def}). 
    Then,
    $$g^{\phi_{T,0,\ell}}_{(0,1)}\leq g_\mu-\frac{\Delta_c}{4(T+1/\beta)}.$$
\end{restatable}
\begin{proof}
We can take a 
value of $\alpha'>0$ such that
$V^{\phi_0}_{\alpha}(i) -  V^{\phi^1_{T,0,\ell}}_{\alpha}(i) \geq \Delta_c/4$ for
$0<\alpha\leq \alpha',i\leq \ell+1$.
This implies for $0<\alpha\leq \alpha', i>\ell$ that
\begin{equation*}
    \begin{split}
        V^{\phi_0}_{\alpha}(i) -  V^{\phi^1_{T,0,\ell}}_{\alpha}(i)&= V^{\ast}_{\alpha}(i,0) -  V^{\ast}_{\alpha}(i,1) -c_r \geq V^{\ast}_{\alpha}(\ell+1,0) -  V^{\ast}_{\alpha}(\ell+1,1) -c_r\\
        &= V^{\phi_0}_{\alpha}(\ell+1,0) -  V^{\phi^1_{T_1,\ell}}_{\alpha}(\ell+1,0) \geq  \frac{\Delta_c}{4},
    \end{split}
\end{equation*}
 where we used the monotonicity of Proposition~\ref{MonotDiff}.

Let $\tau_0=0$ and let $\tau_j$ be the time the $j$-th control is lost for $j\geq 1$ while using policy $\phi^k_{T,0,\ell}$ with $k\geq j$.  Furthermore, we introduce probability distributions $\pi^{i,j}_{t_j}$ over the number of customers  in the system when starting in state $i\in S$ under any policy $\phi^k_{T_1,\ell}$ with $k\geq j$ at $\tau_j$ conditional on $\tau_j=t_j$.
  We let $\phi^0_{T,0,\ell} =\phi_0$ and note that the costs of the processes under policies $\phi^j_{T,0,\ell} $ and $\phi^{j+1}_{T,0,\ell} $ coincide up to time $\tau_j$.
    Let $0<\alpha\leq \alpha'$. As $\tau_j$ diverges to infinity almost surely as $j\rightarrow \infty$, we note that $V^{\phi^{j}_{T,0,\ell}}_{\alpha}(i) \rightarrow V^{\phi_{T,0,\ell}}_{\alpha}(i) $ as $j\rightarrow \infty$ by the monotone convergence theorem.
    We now consider a telescoping sum over the number of repairs. 
   
    \begin{align}\label{eq:tele1}
         V^{\phi_0}_{\alpha}(i) - V^{\phi_{T,0,\ell}}_{\alpha}(i) &= \sum_{j=0}^{\infty} \Big(V^{\phi^{j}_{T,0,\ell}}_{\alpha}(i) -  V^{\phi^{j+1}_{T,0,\ell}}_{\alpha}(i)\Big)\nonumber  \\
          &= \sum_{j=0}^{\infty} \Big(  \mathbb{E}\Big[  \sum_{m=0}^\infty e^{-\alpha \tau_{j}}\cdot \pi^{i,j}_{\tau_{j}}(m)\Big(  V^{\phi_0}_{\alpha}(m) -  V^{\phi^1_{T,0,\ell}}_{\alpha}(m)\Big)   \Big] \Big)\nonumber  \\
          &\geq \sum_{j=0}^{\infty} \Big(  \mathbb{E}\Big[ e^{-\alpha \tau_{j}} \sum_{m=0}^\infty \pi^{i,j}_{\tau_{j}}(m)\cdot \frac{\Delta_c}{4}  \Big] \Big)\nonumber  \\
          &= \sum_{j=0}^{\infty} \Big(  \mathbb{E}\Big[ e^{-\alpha \tau_{j}}\cdot \frac{\Delta_c}{4}  \Big] \Big) =\frac{\Delta_c}{4} \cdot \sum_{j=0}^{\infty} \mathbb{E}\Big[ e^{-\alpha \tau_{j}}  \Big] ,   
    \end{align}
    where in the second equaliy, we used the fact that the sample paths of the processes under $\phi^j_{T,0,\ell}$ and $ \phi^{j+1}_{T,0,\ell}$ coincide up to time $\tau_j$.

  Let $\sigma_j$ be the length of the $j$-th control period. Before each control period, there might be $T$ time units without control.  Take  $\mathcal{S}_j=j T+\sum_{m=1}^j \sigma_m$ as an upper bound of $\tau_j$.
    Note that the length of the control periods are independent, therefore 
   \begin{equation}\label{laplacestieltjes}
       \mathbb{E}\left[ e^{-\alpha \mathcal{S}_j}\right] =e^{-\alpha jT}\Pi_{m=1}^j\mathbb{E}\left[ e^{-\alpha \sigma_m}\right]= e^{-\alpha jT}\mathbb{E}\left[ e^{-\alpha \sigma_1}\right]^j.
        \end{equation}
The Laplace-Stieltjes transform of the exp$(\beta)$ distributed control period length at $\alpha$  is given by
\begin{equation}\label{laplacestieltjes2}
    \mathbb{E}\left[ e^{-\alpha \sigma_1}\right]=\frac{\beta}{\beta+\alpha}.
\end{equation}

Combining Equations~(\ref{laplacestieltjes}) and (\ref{laplacestieltjes2}) we find that

\begin{align}\label{eq:SomSjes}
\begin{split}
     &
     \frac{\Delta_c}{4} \cdot \sum_{j=0}^{\infty} \mathbb{E}\Big[ e^{-\alpha \mathcal{S}_{j}}  \Big]
     = \frac{\Delta_c}{4}\sum_{j=0}^\infty \Big(e^{-\alpha T}\cdot \frac{\beta}{\beta+\alpha}\Big)^j =  \frac{\Delta_c}{4}\cdot \frac{1}{1-e^{-\alpha T}\cdot \beta/(\beta+\alpha)},
\end{split}
\end{align}
as $e^{-\alpha T}\cdot \beta/(\beta+\alpha)<1$.

We combine Equation~(\ref{eq:tele1}) 
 with Equation~(\ref{eq:SomSjes})
 resulting in 

    \begin{equation*} \label{eq:phioneindig}
    \begin{split}
    &\lim_{n \rightarrow \infty}  \alpha\Big(   V^{\phi_0}_{\alpha}(i) -  V^{\phi_{T,0,\ell}}_{\alpha}(i)\Big)    \\
    &\quad \quad \geq \lim_{\alpha \downarrow 0} \alpha
      \frac{\Delta_c}{4} \cdot \sum_{j=0}^{\infty}  \mathbb{E}\Big[ e^{-\alpha S_{j}}  \Big]
    = \frac{\Delta_c}{4}\cdot \lim_{\alpha \downarrow 0} \frac{\alpha}{1-(e^{-\alpha T}\cdot \beta/(\beta+\alpha))}\\
    &\quad \quad=\frac{\Delta_c}{4}\cdot \lim_{\alpha \downarrow 0} \frac{\alpha(\beta+\alpha)}{(1- e^{-\alpha T})\beta+\alpha}\geq \frac{\Delta_c}{4}\cdot  \lim_{\alpha \downarrow 0} \frac{\alpha\beta}{(1- e^{-\alpha T})\beta+\alpha}\\
    &\quad \quad= \frac{\Delta_c}{4}\cdot \lim_{\alpha \downarrow 0} \frac{\beta}{\beta(1- e^{-\alpha T})/\alpha+1} \\
    &\quad \quad= \frac{\Delta_c}{4}\cdot  \frac{\beta}{T\beta+1}=\frac{\Delta_c}{4(T+1/\beta)}.
    \end{split}
   \end{equation*}

    Combined with Corollary~\ref{taubercor} and the fact that 
$ \lim_{\alpha  \downarrow 0}   \alpha V^{\phi_0}_{\alpha}(i)=g_\mu$, we get that the average expected cost is given by \begin{equation*}
    \begin{split}
    &g^{\phi_{T,0,\ell}}_{(0,1)}=\lim_{\alpha \downarrow 0}  \alpha V^{\phi_{T,0,\ell}}_{\alpha}(i)\leq \lim_{\alpha \downarrow 0}  \alpha    V^{\phi_0}_{\alpha}(i) -  \frac{\Delta_c}{4}\cdot  \frac{\beta}{T\beta+1}=g_\mu-\frac{\Delta_c}{4(T+1/\beta)} .
    \end{split}
   \end{equation*}
   
\end{proof}

\subsection{Fixed rate $\mu=\mu_2$}\label{sect:mu2}

Next, we consider the case $\mu=\mu_2$ (but not necessarily $\mu_1>\lambda$).
First we have to determine an appropriate value for  $\ell$.
Recall, that by virtue of Proposition~\ref{ImprovSet}, $H^\ast_\infty(i,0)-H^\ast_\infty(i,1)\geq c^\ast_r $ for $i\leq \ell$ for $\ell=0$. However, if possible, it is preferable to choose a larger value of $\ell$ with this property, to reduce delays before saving costs by regaining control. This motivates to study policies of the form $\phi_{0,T,\ell}$.

First, we define history dependent policies $\phi_{0,T,\ell}^{j}$ that conduct exactly $j\in \mathbb{N}$ repairs by paying $c_r$ for a control period after waiting $T>0$ time units in states $i> \ell$, and directly paying for control in states $i\leq \ell$. 
 During control, we use control policy $\tilde{\phi}_c$.
 
We will show that, as the critical gap is assumed to be positive, for any $m\in \mathbb{N}_0$ there exists a time $T_m\geq 0$, such that $\phi_{0,T_m,\ell}^{1}$ gives an improvement over not regaining control at all, for all starting states $i\leq m$. 
The truncation level $m$ separates states where improvements are guaranteed from rare large queue lengths, the probability of which can be bounded geometrically.

\begin{restatable}{theorem}{TMu2}\label{mufismu2}
   Let $m\in \mathbb{N}_0$  with \begin{equation}\label{eq:Tm}
        T_m = \frac{\log(6 U_0(1+\rho_\mu^{-m/2})/\Delta_c)}{(\sqrt{\mu}-\sqrt{\lambda})^2},
    \end{equation}
    where $U_0$ is defined in Equation~(\ref{eq:Ukdef}). Then,
     \begin{equation*}
    \lim_{\alpha \downarrow 0} \Big(   V^{\phi_0}_{\alpha}(i) -  V^{\phi^1_{0,T_m,\ell}}_{\alpha}(i)\Big)\geq \frac{\Delta_c}{3},
   \end{equation*}
    for all starting states $i\in S$ with $i\leq m$.
\end{restatable}
\begin{proof}
Theorem~\ref{convergenceVI2} yields 
    \begin{align*}
    \lim_{\alpha \downarrow 0} \Big(   V^{\phi_0}_{\alpha}(i) -  V^{\phi^1_{0,T_m,\ell}}_{\alpha}(i)\Big) &=  \lim_{\alpha \downarrow 0} \Big(   V^\ast_{\alpha}(i,0) -  V^\ast_{\alpha}(i,1)\Big) -c_r \\
    &= H^\ast_\infty(i,0)-H^\ast_\infty(i,1)-c_r
    \geq c^\ast_r-c_r = \Delta_c,\quad i\leq \ell.
   \end{align*}

   Take $k:=\tilde{k}_{2\Delta_c/3}$ from Equation~(6.2) of \cite{Koopmans}. Then, by Theorem~6.2 of \cite{Koopmans} $$\sum_{i=0}^k \pi_\mu(i)(H^\ast_\infty(i,0)-H^\ast_\infty(i,1)) \geq c^\ast_r-\frac{\Delta_c}{3}.$$
   For states $\ell<i\leq m$, this implies
   \begin{align} \label{eq:mufmu2p1}
   \begin{split}
    &\lim_{\alpha \downarrow 0} \Big(   V^{\phi_0}_{\alpha}(i) -  V^{\phi^1_{0,T_m,\ell}}_{\alpha}(i)\Big) =  \lim_{\alpha \downarrow 0} e^{-\alpha T_m}\sum_{j=0}^\infty P_{ij}(T_m)  \Big(V^{\ast}_{\alpha}(j,0) -  V^\ast_{\alpha}(j,1)\Big) -c_r \\
    &\quad\geq \lim_{\alpha \downarrow 0} \sum_{j=0}^k P_{ij}(T_m)  \Big(V^{\ast}_{\alpha}(j,0) -  V^\ast_{\alpha}(j,1)\Big) -c_r = \sum_{j=0}^k P_{ij}(T_m)  \Big(H^\ast_\infty(j,0) -  H^\ast_\infty(j,1)\Big) -c_r \\
   &\quad\geq \sum_{j=0}^k \pi_\mu(j)  \Big(H^\ast_\infty(j,0) -  H^\ast_\infty(j,1)\Big)- \sum_{j=0}^k \Big |P_{ij}(T_m)-\pi_\mu(j)\Big |  \Big(H^\ast_\infty(j,0) -  H^\ast_\infty(j,1)\Big) -c_r \\
  &\quad\geq c^\ast_r- \frac{\Delta_c}{3} -c_r- \sum_{j=0}^k \Big |P_{ij}(T_m)-\pi_\mu(j)\Big |  \Big(H^\ast_\infty(j,0) -  H^\ast_\infty(j,1)\Big) \\
  &\quad\geq  \frac{2\Delta_c}{3} - \sum_{j=0}^k \Big |P_{ij}(T_m)-\pi_\mu(j)\Big |  \Big(H^\ast_\infty(j,0) -  H^\ast_\infty(j,1)\Big),
     \end{split}
   \end{align}
where we used non-negativity of $H^\ast_{\infty}(j,0)-H_\infty^\ast(j,1)$ (cf. Proposition~\ref{MonotDiff}) in the third line.
   
   Next, we bound the sum $\sum_{j=0}^k \Big |P_{ij}(T_m)-\pi_\mu(j)\Big |  \Big(H^\ast_\infty(j,0) -  H^\ast_\infty(j,1)\Big)$.
   By Proposition~\ref{MonotDiff} \\
   $0\leq H^\ast_\infty(i,0) -  H^\ast_\infty(i,1)\leq H^\ast_\infty(0,0) -  H^\ast_\infty(0,1)$, for $i\leq k$. Using that $U_0$ is an upper bound of $H^\ast_\infty(0,0) -  H^\ast_\infty(0,1)$
   and plugging in Proposition~\ref{TVexpconverg} and our choice of $T_m$ gives
   \begin{equation}\label{eq:mufmu2p2}
   \begin{split}
       &\sum_{j=0}^k \Big |P_{ij}(T_m)-\pi_\mu(j)\Big |  \Big(H^\ast_\infty(j,0) -  H^\ast_\infty(j,1)\Big)
       \leq U_0 \sum_{j=0}^k \Big |P_{ij}(T_m)-\pi_\mu(j)\Big |  \\
       &\quad \quad\quad \quad\quad \quad\quad \quad\quad \quad\quad \quad\leq U_0 \sum_{j=0}^\infty \Big |P_{ij}(T_m)-\pi_\mu(j)\Big | \leq 2 U_0(1+\rho_\mu^{-m/2})e^{-(\sqrt{\mu}-\sqrt{\lambda})^2 T_m} \\
       &\quad \quad\quad \quad\quad \quad\quad \quad\quad \quad\quad \quad\leq 2 U_0(1+\rho_\mu^{-m/2})e^{-(\sqrt{\mu}-\sqrt{\lambda})^2 T_m} =\frac{\Delta_c}{3}.
        \end{split}
   \end{equation}

   Combining Equations~(\ref{eq:mufmu2p1}) and (\ref{eq:mufmu2p2}) then gives
   \begin{align*} 
       \lim_{\alpha \downarrow 0} \Big(   V^{\phi_0}_{\alpha}(i) -  V^{\phi^1_{0,T_m,\ell}}_{\alpha}(i)\Big)\geq   \frac{2\Delta_c}{3}-\frac{\Delta_c}{3}=  \frac{\Delta_c}{3}.
   \end{align*}
\end{proof}



Next, we analyse the tail probabilities of the process $X(t)$ under policy $\phi_{0,T_m,\ell}$. For convenience, we track the number of customers $I(t)$ at time $t\geq 0$.
Recall
that control policy  $\tilde{\phi}_c$  uses $\mu_2$ exactly in the states $(i,1) $ with $i > i^\ast$. 

\begin{proposition} \label{boundedtailmu2}
    Let the process 
    under policy $\phi_{0,T_m,\ell}$, start in $(0,1)\in S$.
    For \begin{equation}\label{eq:Adef}
        A=\max\Big\{\rho_2^{-i^\ast},\Big( \frac{1-\rho_1^{i^\ast}}{1-\rho_1}+\frac{\rho_1^{i^\ast}}{1-\rho_2} \Big)^{-1}  \cdot \rho_1^{i^\ast}\cdot\frac{\rho_2^{-i^\ast}}{1-\rho_2}\Big\} 
    \end{equation} 
    it holds that  $\mathbb{P}(I(t)\geq j)\leq A \cdot \rho_2^{j}$\, for  $j\in \mathbb{N}_0$ and $t\geq 0$, independently of $T_m$ and $\ell$.
\end{proposition}
\begin{proof}
At time $t=0$ all bounds hold. Furthermore, we note that the value of $I(t)$ is stochastically smaller than the number of customers $\tilde{I}(t)$ of the (birth-death)  process that always uses the service rate specified by $\tilde{\phi}_c$  (as this replaces the service rate $\mu=\mu_2$ in underlying states $(i,0,t)$ by 
$\mu_{\tilde{\phi}_c(i)} \leq \mu$).
By Keilson and Kester~\cite{KEILSON1977231}, 
$\tilde{I}(t)$ is stochastically smaller than its limit $\lim_{u\rightarrow \infty}\tilde{I}(u)$ with stationary distribution $\tilde{\pi}$, where 
\begin{equation}\label{eq:pitilde}
   \tilde{\pi}(i)=\begin{cases}
 \Big( \frac{1-\rho_1^{i^\ast}}{1-\rho_1}+\frac{\rho_1^{i^\ast}}{1-\rho_2} \Big)^{-1}  \cdot \rho_1^{i} \ \ \ \text{ for } i\leq i^\ast,\\
  \Big( \frac{1-\rho_1^{i^\ast}}{1-\rho_1}+\frac{\rho_1^{i^\ast}}{1-\rho_2} \Big)^{-1}  \cdot \rho_1^{i^\ast}\cdot \rho_2^{i-i^\ast} \ \ \ \text{ for } i > i^\ast.
\end{cases} 
\end{equation}

Consequently, 
\begin{align*}
    \mathbb{P}(I(t)\geq j) &\leq \mathbb{P}(\tilde{I}(t)\geq j) \leq \sum_{i=j}^{\infty} \tilde{\pi}(i)=\Big( \frac{1-\rho_1^{i^\ast}}{1-\rho_1}+\frac{\rho_1^{i^\ast}}{1-\rho_2} \Big)^{-1}  \cdot \rho_1^{i^\ast}\cdot \sum_{i=j}^\infty\rho_2^{i-i^\ast}\\
    &=\Big( \frac{1-\rho_1^{i^\ast}}{1-\rho_1}+\frac{\rho_1^{i^\ast}}{1-\rho_2} \Big)^{-1}  \cdot \rho_1^{i^\ast}\cdot\frac{\rho_2^{j-i^\ast}}{1-\rho_2}=\rho_2^j\cdot\Big( \frac{1-\rho_1^{i^\ast}}{1-\rho_1}+\frac{\rho_1^{i^\ast}}{1-\rho_2} \Big)^{-1}  \cdot \rho_1^{i^\ast}\cdot\frac{\rho_2^{-i^\ast}}{1-\rho_2},
\end{align*}
for $j\geq i^\ast$.

For $j<i^\ast$, we see that trivially $\mathbb{P}(I(t)\geq j) \leq 1\leq \rho_2^{j-i^\ast} = \rho_2^j\cdot \rho_2^{-i^\ast} $.

Combining these, gives that $$\mathbb{P}(I(t)\geq j) \leq \rho_2^j\cdot \max\Big\{\rho_2^{-i^\ast},\Big( \frac{1-\rho_1^{i^\ast}}{1-\rho_1}+\frac{\rho_1^{i^\ast}}{1-\rho_2} \Big)^{-1}  \cdot \rho_1^{i^\ast}\cdot\frac{\rho_2^{-i^\ast}}{1-\rho_2}\Big\}, \quad j\in \mathbb{N}_0.$$ 

\end{proof}

Conditioning on being under control increases the above tail bound by at most an exponential factor, because control periods become less likely as $T_m$ increases. To analyse this, we also need the track whether the process is in control or not at time $t\geq 0$, denoted by $S(t)=1$ and $S(t)=0$, respectively.

\begin{proposition} \label{propCor}
   Let the process 
    under policy $\phi_{0,T_m,\ell}$, start in $(0,1)\in S$.
    It holds that  
    $$\mathbb{P}(I(t)\geq j\ | \ S(t)=1)\leq A_{T_m} \cdot \rho_2^{j}
    $$ 
    for all queue lengths 
    $j\in \mathbb{N}_0$ with $j\geq \ell$ and any time $t\geq 0$ for $A_{T_m}=Ae^{\beta \cdot T_m}$,
    with $A$  defined in Equation~(\ref{eq:Adef}). 
\end{proposition}
\begin{proof}
In order to prove this proposition, we will combine a bound on $\mathbb{P}(S(t)=1)$ over all $t\in \mathbb{R}_{\geq 0}$, with the tail bound found in Proposition~\ref{boundedtailmu2}. The latter bound is not dependent on the choices of $\ell,m$ and corresponding $T_m$.

For any $t\in \mathbb{R}_{\geq 0}$ such that there has not been a loss of control since time $(t-T_m)^+$, it holds that $S(t)=1$. The former happens with a probability of at least $e^{-\beta T_m}$. As such, we see that $\mathbb{P}(S(t)=1) \geq e^{-\beta T_m}$. 

 For the conditional tail bound, we note that  \begin{equation*}
     \mathbb{P}(I(t)\geq j\ | \ S(t)=1)=\frac{\mathbb{P}(I(t)\geq j  , S(t)=1)}{\mathbb{P}(S(t)=1)}\leq \frac{ \mathbb{P}(I(t)\geq j )}{e^{-\beta \cdot T_m}}.
 \end{equation*}
The geometric tail bound follows directly from Proposition~\ref{boundedtailmu2}.
\end{proof}

Also for policies $\phi_{0,T_m,\ell}$, the process  regenerates at return times to state $(0,1)$. Consequently, we can again apply the Tauberian theorem through a renewal argument.

\begin{restatable}{proposition}{RenRew2}\label{renrewmu2}
For any value of  $ T_m\in \mathbb{R}$, the (non-discounted) process under policy $\phi_{0,T_m,\ell}$ induces a renewal reward process defined  by random variables $S_i,W_i$, $i\in \mathbb{N}_{\geq 1}$, that denote the time and costs between the $(i-1)$-th and $i$-th entrance in state $(0,1)$, respectively. 
For this renewal reward process both $\mathbb{E}[S_1]<\infty$ and $\mathbb{E}[|W_1|]<\infty$.
\end{restatable}
\begin{proof}
 Clearly, this is clearly a renewal reward process.  
 We want to find an upper bound of the expected time of the process under policy $\phi_{0,T_m,\ell}$ before entering state $(0,1)$, starting in state $(0,1)$ given by $\mathbb{E}[S_1]$. We do this by considering the busy periods again. 
 
   The first jump, starting in state $(0,1)$, occurs after an expected time of $1/(\lambda+\beta)$.  If it is a sensor breakdown, we jump directly to state $(0,1)$ (due to the instant repair). When it is an arrival, the first busy period starts.
   
   We note that in states with customer amounts that are larger than $i^\ast$, we always use the service rate $\mu_2$. For states with a customer amount less than, or equal to $i^\ast$, we either use service rate $\mu_1$ or $\mu_2$ depending on whether we have control. Therefore, the expected length of a busy period is at most the expected busy period length of the birth-death process $\tilde{I}(t)$ from the previous two propositions.
   
   
   We note that the expected duration of this process in state $0$ is given by $1/\lambda$.
   Consequently, $$\tilde{\pi}(0)=\frac{1/\lambda}{\mathbb{E}[BP]+1/\lambda} \implies \mathbb{E}[BP]=\Big( \frac{1-\rho_1^{i^\ast}}{1-\rho_1}+\frac{\rho_1^{i^\ast}}{1-\rho_2} \Big)\Big/\lambda-\frac{1}{\lambda}. $$  
   
   Now, we consider the moment that a busy period ends in our original process under policy $\phi_{0,T_m,\ell}$. Either we are in state $(0,1)$ or in a state $(0,0,t)$ with $t\leq T_m$. This means that with a probability of at least $e^{-\lambda T_m}$, the next jump is to state $(0,1)$ after $t\leq T_m$ time and otherwise the next busy period commences within $T_m$ time. The expected amount of busy periods needed before reaching state $(0,1)$ is bounded from above by the expectation of a geometrically distributed random variable with success probability $e^{-\lambda T_m}$. This expectation is given by $e^{\lambda T_m}$. 
   Therefore,
   $$\mathbb{E}[S_1] \leq \frac{1}{\lambda+\beta} + e^{\lambda T_m} \Big(T_m+\Big( \frac{1-\rho_1^{i^\ast}}{1-\rho_1}+\frac{\rho_1^{i^\ast}}{1-\rho_2} \Big)\Big/\lambda-\frac{1}{\lambda}\Big)<\infty.$$
   
   To bound $\mathbb{E}[|W_1|]$ from above, we consider the costs for service rate $\mu_2$ and the holding costs separately.
   The expected costs for service rate $\mu_2$ per renewal period are bounded from above by $c_{\mu_2}\cdot \mathbb{E}[S_1] $. 
   
   We note that the holding costs are zero when not in a busy period. As such, we consider the holding costs per busy period.
As the holding cost is a non-decreasing function,   the expectation of the holding costs in a busy period is at most the expectation of the holding costs in a busy period of the process $\tilde{I}(t)$.
   
   This birth-death process has a finite average expected cost as follows from Assumption~\ref{holdingassump} and it is given by  $\tilde{g}_h=\sum_{i=0}^\infty \tilde{\pi}(i)\cdot h(i) <\infty$. 
   The expected duration per renewal $\mathbb{E}[\tilde{S}]$ of the birth-death process is given by the expected time until the first arrival when the busy period starts, and the expected time of that busy period $\mathbb{E}[BP]$. This is $$\mathbb{E}[\tilde{S}]= \frac{1}{\lambda}+ \Big( \frac{1-\rho_1^{i^\ast}}{1-\rho_1}+\frac{\rho_1^{i^\ast}}{1-\rho_2} \Big)\Big/\lambda-\frac{1}{\lambda}=\Big( \frac{1-\rho_1^{i^\ast}}{1-\rho_1}+\frac{\rho_1^{i^\ast}}{1-\rho_2} \Big)\Big/\lambda.$$ 
   We can use the renewal theorem to find the expected holding costs per period, which is equal to 
   $$\tilde{g}_h \cdot \Big( \frac{1-\rho_1^{i^\ast}}{1-\rho_1}+\frac{\rho_1^{i^\ast}}{1-\rho_2} \Big)\Big/\lambda.$$ Next, we consider that the expected amount of busy periods before reaching $(0,1)$ in the process under $\phi_{0,T_m,\ell}$ is at most $e^{\lambda T_m}$ and we can bound  $\mathbb{E}[|W_1|]$ by 
   $$\mathbb{E}[|W_1|]\leq e^{\lambda T_m}\cdot \tilde{g}_h\cdot  \Big( \frac{1-\rho_1^{i^\ast}}{1-\rho_1}+\frac{\rho_1^{i^\ast}}{1-\rho_2} \Big)\Big/\lambda<\infty.$$
  This concludes the proof.
  
\end{proof}

Now we have all the ingredients to show that there is an average expected cost reduction under policy $\phi_{0,T_m,\ell}$ w.r.t. policy $\phi_0$. 

\begin{restatable}{theorem}{TMu2avg}\label{mufismu2avg}
Let $$m= \left\lceil\max\left\{2\log_{\rho_2}(\Delta_c/(12A'c_r)),2\log_{\rho_2}(1/(3A'))\right\}\right\rceil,$$
where $A'= 18e^{\beta/(\sqrt{\mu_2}-\sqrt{\lambda})^2}  U_0A/\Delta_c $ with $A$ given by Equation~(\ref{eq:Adef}). 
    Then,
    $$g^{\phi_{0,T_m,\ell}}_{(0,1)}\leq g_\mu-\frac{\Delta_c}{12(1/\beta+ T_m)},$$
     with $T_m$ given in Equation~(\ref{eq:Tm}).
\end{restatable}
\begin{proof}
    We start with an analogous line of reasoning as in the proof of Theorem~\ref{mufismu1avg} with a sufficient choice of $\alpha'$ and considering $0<\alpha\leq \alpha'$. We can directly adopt the first half of Equation~(\ref{eq:tele1}) (with $i=0$) for the sequence  $\{\phi^k_{0,T_m,\ell}\}_{k\in \mathbb{N}}$, such that 

\begin{equation}
\begin{split}\label{eq:tele1mu2}
         V^{\phi_0}_{\alpha}(0) -  V^{\phi_{0,T_m,\ell}}_{\alpha}(0)
          &= \sum_{j=0}^{\infty} \Big(  \mathbb{E}\Big[  \sum_{i=0}^\infty e^{-\alpha \tau_{j}}\cdot \pi^{0,j}_{\tau_{j}}(i)\Big(  V^{\phi_0}_{\alpha}(i) -  V^{\phi^1_{0,T_m,\ell}}_{\alpha}(i)\Big)   \Big] \Big)\\ 
          &\geq \sum_{j=0}^{\infty} \Big(  \mathbb{E}\Big[  \sum_{i=0}^m e^{-\alpha \tau_{j}}\cdot \pi^{0,j}_{\tau_{j}}(i)\Big(  V^{\phi_0}_{\alpha}(i) -  V^{\phi^1_{0,T_m,\ell}}_{\alpha}(i)\Big)   \Big] \\
          &\hspace{5.3cm}- c_r\cdot \mathbb{E}\Big[ \sum_{i=m+1}^\infty  e^{-\alpha \tau_{j}}\cdot \pi^{0,j}_{\tau_{j}}(i) \Big] \Big) \\
          &\geq \sum_{j=0}^{\infty} \Big(  \frac{\Delta_c}{4}\cdot \mathbb{E}\Big[  e^{-\alpha \tau_{j}}\cdot \sum_{i=0}^m \pi^{0,j}_{\tau_{j}}(i)  \Big] - c_r \cdot \mathbb{E}\Big[  e^{-\alpha \tau_{j}}\cdot\sum_{i=m+1}^\infty  \pi^{0,j}_{\tau_{j}}(i) \Big] \Big)\\
           &= \sum_{j=0}^{\infty} \Big(  \mathbb{E}\Big[  e^{-\alpha \tau_{j}}\cdot \Big((\frac{\Delta_c}{4}+c_r)\sum_{i=0}^m \pi^{0,j}_{\tau_{j}}(i) -c_r\Big] \Big). 
    \end{split}     
    \end{equation}

    As the number of customers in the system at time $\tau_j$ is not independently distributed with respect to the random variable $\tau_j$ itself,
 it is not directly possible to use Proposition~\ref{boundedtailmu2}. 

 To solve this, we have to calculate two measures: the distribution of $\tau_j$ and the distribution  of $\pi^{0,j}_{\tau_{i}}$.

Let  
$P_j$ denote the distribution of $\tau_j$. This is a distribution on $[0,\infty)$, noting that the first control period starts at time $0$.
For convenience we, write $F_j(x)=P_{j} [0,x)$, $x\geq 0 $.
Then, similarly to the proof of the Lipschitz continuity of the function $\mathbb{E}[ N(t)]$, it is simply checked that $F$ is Lipschitz continuous.
By Billingsley\cite[Thm 31.8]{billingsley}, $F_j(t)=\int_{0}^x p((s)ds$, for a Lebesgue-integrable function $p:[0,\infty)\to[0,\infty)$.
By Billingsley\cite[Thm 31.3]{billingsley},  $p=F_j'$ except on a set of Lebesgue measure 0, i.o.w., it is almost surely equal to the right derivative of $F_j$, i.e.
\begin{equation}
\label{maat1}
p(t)\stackrel{a.s.}{=}\lim_{\delta\downarrow0} \frac{P_{j}[t,t+\delta)}{\delta}=\lim_{\delta\downarrow 0}  \frac{\mathbb{P}(\tau_j\in [t,t+\delta ))}{\delta}.
\end{equation}

For $\pi^{0,j}_{\tau_j}(i)$, we take an analogous approach by  considering the distribution function
$$
F(t)=\int_{0}^t\pi^{0,j}_s dP_{j}(s)=\int_{0}^t \mathbb{P}(I(\tau_j)=i\,|\, \tau_j=s) dP_{j}(s)=\int_0^t  \mathbb{P}(I(\tau_j)=i\,|\, \tau_j=s) p(s)ds.
$$
Then, again by \cite[Thm 31.3]{billingsley}, $F'(t)= \mathbb{P}(I(\tau_j)=i\,|\, \tau_j=t) p(t)$, except on a set of Lebesgue measure 0.

Notationally, it is convenient to use that the processes are implicitly defined on an underlying space $(\Omega,{\cal F},\mathbb{P})$.
Then, $P_J$ is the induced probability distribution of $\tau_j$. Let $t>0$. We get, 
\begin{equation}\label{maatpitauj}
\begin{split}
\mathbb{P}(I(\tau_j)=i\,|\, \tau_j=t) p(t)&\stackrel{a.s.}{=}\lim_{\delta\downarrow 0}\frac{F(t+\delta)-F(t)}{\delta}=\lim_{\delta\downarrow0}\frac{\int_t^{t+\delta} \mathbb{P}(I(\tau_j)=i\,|\,
\tau_j=s)dP_J(s)}{\delta}\\
&=\lim_{\delta\downarrow 0}\frac{\int_{\tau_j\in[t,t+\delta)}\mathbb{E}({\bf 1}_i(I(\tau_j)\,|\, \tau_j\in[t,t+\delta))d\mathbb{P}}{\delta}\\
&=\lim_{\delta\downarrow 0}\frac{\int_{\tau_j\in [t,t+\delta)}{\bf 1}_i(I(\tau_j))d\mathbb{P}}{\delta}\\
&=\lim_{\delta\downarrow 0}\frac{\int{\bf 1}(\tau_j\in [t,t+\delta))\cdot{\bf 1}_i(I(\tau_j))d\mathbb{P}}{\delta}\\
&=\lim_{\delta\downarrow 0}\frac{\mathbb{P}\big( I(\tau_j)=i, \tau_j\in[t,t+\delta))}{\delta}\\
&=\lim_{\delta\downarrow0}\frac{\mathbb{P}\big( I(\tau_j)=i, \tau_j\in[t,t+\delta))}{\mathbb{P}(\tau_j\in [t,t+\delta))}\cdot \frac{\mathbb{P}(\tau_j\in [t,t+\delta))}{\delta}.
\end{split}
\end{equation}

In order to calculate $p$ and $\pi^{0,j}_{\tau_j}(i)$ a.s., we introduce $U(t)$ to denote the number of control periods that have started at time $t\geq 0$. 
The number of customers in the system and control losses depend on exponential distributions and the probability of at least two exponentially distributed events within a time interval of length $\delta t$ is of order $\mathcal{O}((\delta t)^2)$.
As a consequence,
\begin{equation*}
    \begin{split}
        \mathbb{P}(\tau_j\in [t,t+\delta t))= \mathbb{P}(\tau_j\in (t,t+\delta t),U(t)=j,S(t)=1)+\mathcal{O}((\delta t)^2).
    \end{split}
\end{equation*}
Furthermore,
\begin{equation*}
    \begin{split}
        \mathbb{P}(  I(\tau_j)=i , \tau_j \in [t,t+\delta t))&=\mathbb{P}(  I(t)=i , \tau_j \in [t,t+\delta t))+\mathcal{O}((\delta t)^2)\\
        &= \mathbb{P}(  I(t)=i , U(t)=j,S(t)=1,\tau_j \in (t,t+\delta t) ) +\mathcal{O}((\delta t)^2) .
    \end{split}
\end{equation*}

Thus, we can use Equation~(\ref{maat1}) to derive

   \begin{equation}
       \begin{split}\label{eq:tjdens}
         p(t)\stackrel{a.s.}{=}   \lim_{\delta t \downarrow 0} \frac{\mathbb{P}(\tau_j\in [t,t+\delta t) )}{\delta t}
            &=\lim_{\delta t \downarrow 0}\frac{\mathbb{P}(\tau_j\in (t,t+\delta t) , U(t)=j, S(t)=1 ) }{\delta t}\\
            &
           = \lim_{\delta t \downarrow 0}\frac{\mathbb{P}(\tau_j\in (t,t+\delta t))\mathbb{P}( U(t)=j, S(t)=1 ) }{\delta t} \\
           &= \beta \cdot\mathbb{P}(U(t)=j,S(t)=1), 
       \end{split}
   \end{equation}
   where we used the independence of the events $\{\tau_j\in(t,t+\delta t)\}$ and $\{U(t)=j,S(t)=1\}$.
   Note that $p(t)\stackrel{a.s.}{>}0$ for $t>0$.
Hence, it remains to calculate $\pi_{t}^{0,j}(i)$ for $t>0$ using Equation~(\ref{maatpitauj}).
\begin{equation} \label{eq:pitprob}
\begin{split}
\pi_{t}^{0,j}(i):=\mathbb{P}(  I(\tau_j)=i \ | \ \tau_j =t) &\stackrel{a.s.}{=}\lim_{\delta\downarrow0}\frac{\mathbb{P}\big( I(\tau_j)=i, \tau_j\in[t,t+\delta))}{\mathbb{P}(\tau_j\in [t,t+\delta))}\\
&=\lim_{\delta t \downarrow 0}\frac{\mathbb{P}(  I(t)=i , U(t)=j,S(t)=1,\tau_j \in (t,t+\delta t) )  }{\mathbb{P}( U(t)=j,S(t)=1,  \tau_j \in (t,t+\delta t))}\\
&=    \lim_{\delta t \downarrow 0}\frac{\mathbb{P}(  I(t)=i , U(t)=j, S(t)=1 )  }{\mathbb{P}( U(t)=j,S(t)=1)} 
              \\
              &=  \mathbb{P}(  I(t)=i \ | \  U(t)=j, S(t)=1 ),
\end{split}
\end{equation}

where, in the second last equality, we use that $I(t)$, given that $U(t)=j,S(t)=1$, is independent of the next moment when control is lost.

  Now, we can use Equation~(\ref{eq:pitprob}) and the density of Equation~(\ref{eq:tjdens}) to continue Equation~(\ref{eq:tele1mu2}) as follows.

   \begin{equation*}
    \begin{split}\label{eq:teleinf2}
         &V^{\phi_0}_{\alpha}(0) -  V^{\phi_{T_m,\ell}}_{\alpha}(0)
         \geq \sum_{j=0}^{\infty} \Big(  \mathbb{E}\Big[  e^{-\alpha \tau_{j}}\cdot \Big((\frac{\Delta_c}{4}+c_r)\sum_{i=0}^m \pi^{0,j}_{\tau_{j}}(i) -c_r\Big) \Big]\\
         &\quad \quad= \sum_{j=0}^{\infty} \Big(  \int_0^\infty  e^{-\alpha t} \cdot \beta \cdot  \mathbb{P}(U(t)=j,S(t)=1) \\
         &\hspace{6cm} \cdot  \Big((\frac{\Delta_c}{4}+c_r)\sum_{i=0}^m \mathbb{P}(  I(t)=i \ | \ U(t)=j,  S(t)=1 ) -c_r\Big) dt\Big)\\
         &\quad \quad= \beta  \int_0^\infty  e^{-\alpha t} \cdot \mathbb{P}(S(t)=1)\\
         &\hspace{2.3cm}\cdot \Big((\frac{\Delta_c}{4}+c_r) \sum_{j=0}^{\infty} \mathbb{P}(U(t)=j\ | \ S(t)=1)\mathbb{P}(  I(t)\leq m \ | \ U(t)=j,  S(t)=1 ) -c_r\Big)  dt \\
         &\quad \quad = \beta   \int_0^\infty  e^{-\alpha t}\cdot \Big((\frac{\Delta_c}{4}+c_r) \mathbb{P}(I(t)\leq m \ | \ S(t)=1) -c_r\Big) \cdot \mathbb{P}(S(t)=1) dt. 
    \end{split}     
    \end{equation*}

We use the chosen values of $m$ and $T_m$ with Proposition~\ref{propCor}, such that 
\begin{equation*}\label{eq:mbound0}
    \mathbb{P}(I(t)>m \ | \ S(t)=1)\leq \frac{A}{e^{-\beta \cdot T_m}}\cdot \rho_2^m = e^{\beta/(\sqrt{\mu_2}-\sqrt{\lambda})^2} \cdot (6 U_0(1+\rho_\mu^{-m/2})/\Delta_c) \cdot A \cdot \rho_2^m.
\end{equation*}
This gives 
\begin{equation}\label{eq:mbound1}
    \mathbb{P}(I(t)>m \ | \ S(t)=1)\leq 
  A'\cdot  \rho_2^{m/2} \leq  A'\cdot \frac{\Delta_c}{12A'c_r} = \frac{\Delta_c}{12c_r}.
\end{equation}
but also that
\begin{equation}\label{eq:mbound2}
    \mathbb{P}(I(t)>m \ | \ S(t)=1)\leq 
     A'\cdot  \rho_2^{m/2} \leq   A'\cdot  \frac{1}{3A'} = \frac{1}{3}\implies \mathbb{P}(I(t)\leq m \ | \ S(t)=1)\geq \frac{2}{3}.
\end{equation}
Using Equations~(\ref{eq:mbound1}) and (\ref{eq:mbound2}),
we obtain 
\begin{equation*}
    \begin{split}
        &V^{\phi_0}_{\alpha}(0) -  V^{\phi_{0,T_m,\ell}}_{\alpha}(0)
        \\
        &\hspace{1cm}\geq \beta  \Big(  \int_0^\infty  e^{-\alpha t}\cdot \Big((\frac{\Delta_c}{4}+c_r) \mathbb{P}(I(t)\leq m \ | \ S(t)=1) -c_r\Big) \cdot \mathbb{P}(S(t)=1) dt \Big)\\
          &\hspace{1cm}\geq  \beta\Big(  \int_{t=0}^\infty e^{-\alpha t}  \cdot \mathbb{P}(S(t)=1)\Big(\frac{\Delta_c}{4}\cdot \frac{2}{3} -c_r\cdot  \frac{\Delta_c}{12 c_r}\Big) dt \Big)\\
           &\hspace{1cm}=\frac{\beta \Delta_c}{12}   \int_{t=0}^\infty e^{-\alpha t}\cdot\mathbb{P}(S(t)=1)dt.
    \end{split}
\end{equation*}Thus, we get
\begin{equation*}
    \lim_{\alpha\downarrow 0}\alpha(V^{\phi_0}_{\alpha}(0) -  V^{\phi_{0,T_m,\ell}}_{\alpha}(0))\geq \frac{\beta \Delta_c}{12}\cdot \lim_{\alpha\downarrow 0}\alpha  \int_{t=0}^\infty e^{-\alpha t}\cdot\mathbb{P}(S(t)=1)dt.
\end{equation*}

 The Tauberian theorem for $S(t)$ implies  that $\lim_{\alpha\downarrow 0}\alpha  \int_{t=0}^\infty e^{-\alpha t}\cdot\mathbb{P}(S(t)=1)dt$ is equal to the fraction of time that the system in control. As every period of control of expected length $1/\beta$ is either followed up by 0 or $T_m$ time units,  the fraction of time that  the process is in control is at least $$\frac{\frac{1}{\beta}}{\frac{1}{\beta}+T_m}.$$ 
 Hence,
 \begin{equation*}
 \begin{split}
      \lim_{\alpha\downarrow 0}\alpha(V^{\phi_0}_{\alpha}(0) -  V^{\phi_{0,T_m,\ell}}_{\alpha}(0))&\geq \frac{\beta \Delta_c}{12}\cdot \lim_{\alpha\downarrow 0}\alpha   \int_{t=0}^\infty e^{-\alpha t}\cdot\mathbb{P}(S(t)=1) dt\\
      &\geq \frac{\beta \Delta_c}{12}\cdot \frac{\frac{1}{\beta}}{\frac{1}{\beta}+T_m}= \frac{\Delta_c}{12(\frac{1}{\beta}+T_m)}.
 \end{split}
\end{equation*}

From the Tauberian theorem for the processes using policies $\phi_0$ and $\phi_{0,T_m,\ell}$ we can conclude that
$$g^{\phi_{0,T_m,\ell}}_{(0,1)}\leq g_\mu-\frac{\Delta_c}{12(1/\beta+ T_m)}.$$

\end{proof}

The bound shows that any positive critical gap yields strictly improving policies, although guarantees may be conservative.


\section{Computational results}\label{sect:comp}

In this section, we will consider three cases of parameter combinations and will find improving policies according to the theory supplied in Section~\ref{sect:improv} for both $\mu=\mu_1$ and $\mu=\mu_2$. We vary the costs through different repair costs $c_r$ (close to the critical cost $c^\ast_r$ and less close) and through the choice of linear holding costs of the form $h(i)=K_{lin}\cdot i$ or quadratic  holding costs of the form $h(i)=K_{sqd}\cdot i^2$.  The three different parameter combinations we consider are ordered by traffic intensity and given in Table~\ref{tablecases}.

\begin{table}[H]
\begin{center}
\caption{The three considered combinations of parameters.}
\label{tablecases}
\begin{tabular}{ |p{1.2cm}|p{0.8cm}|p{0.8cm}|p{0.8cm}|p{0.8cm}|p{0.8cm}| }
 \hline
  Cases&$\beta$ & $\lambda$ & $\mu_1$ & $\mu_2$ & $c_{\mu_2}$  \\
 \hline
  Case 1 & 0.1    & 0.1&   0.35 & 0.45 & 10  \\
 Case 2  &0.05    & 0.2&   0.35 & 0.4 & 10  \\
 Case 3    & 0.02    & 0.31&   0.33 & 0.34 & 10  \\
 \hline
\end{tabular}
\end{center}
\end{table}


We will find the improving policies by giving the Threshold $i^\ast$  of control policy $\tilde{\phi}_c$ together with $\ell$ and $T_1$ when $\mu=\mu_1$ and together with $\ell,m$ and $T_m$ when $\mu=\mu_2$.
We also give the theoretic lower bound of the average expected cost reduction. All these results come from an extension of the Python code used for \cite{Koopmans}. Next, we will use simulation to estimate the average expected cost of the improving policies following Asmussen and Glynn\cite{asmussen}. Finally, we will conclude by giving approximate confidence intervals of the average expected cost.

We note that the average expected cost under the renewal reward processes induced by the improving policies are equal to $\mathbb{E}[W]/\mathbb{E}[S]$ (which we simply refer to as the average cost).
Following Proposition~4.1 of Asmussen and Glynn~\cite{asmussen} , we can take the sample ratio given by
\begin{equation*}
    \hat{g}= \frac{\Bar{W}}{\Bar{S}}.
\end{equation*}
 as  estimator of the average cost. Subsequently, we can derive approximate 99\% confidence intervals as

 \begin{equation*}
     \hat{g}\pm \frac{z_{0.995}\hat{\eta}}{\sqrt{n}},
 \end{equation*}
 where $z_{0.995}$ is the $0.995$-quantile of $\mathcal{N}(0,1)$ and where $$\hat{\eta}^2=\frac{1}{n-1}\sum_{i=1}^n (W_i-\hat{g}S_i)^2\Big/ \Big(\frac{1}{n}\sum_{i=1}^n S_i\Big)^2.$$

\subsection{ Saved average cost when $\mu=\mu_1$}

The results for $\mu=\mu_1$ with linear holding cost function $h(i)=5\cdot i$ are all summarised in Table~\ref{table0lin}.

\begin{table}[H]
\begin{center}
\caption{Lower bounds of $\Delta_c$, thresholds $i^\ast$ and values of $\ell$ and  $T$ for improving policies and corresponding lower bounds of average cost reduction and estimated average cost reductions with corresponding $99\%$ confidence intervals for various instances with linear holding costs $h(i)=5\cdot i$ and $\mu=\mu_1$.}
\label{table0lin}
\begin{tabular}{ |p{1.1cm}|p{0.8cm}|p{1.2cm}|p{0.6cm}|p{0.6cm}|p{0.6cm}|p{1.5cm} |p{1.8cm} |p{1.8cm} |p{2.5cm} | }
 \hline
 \multicolumn{2}{|c|}{Case and $c_r$} & \multicolumn{8}{|c|}{ Results}\\
 \hline
  Case & $c_r$   &  $\Delta_c$ & $g_\mu$ & $i^\ast$ & $\ell$ & $T$ & Lower bound & Avg. cost saved & 99\% Confidence interval\\
 \hline
   Case 1 &  0.01 & 0.0008  &  2 & 5 
  &2& 
  248.732   &  $7.730\cdot 10^{-7}$ &  $1.616\cdot 10^{-4}$
  & $[1.9992,2.0005]$\\
     Case 1 &  0.005 & 0.0058  &  2 & 5 
  &2& 
  221.810   &  $6.255\cdot 10^{-6}$ & $9.290\cdot 10^{-5}$  & $[1.9992,2.0006]$\\
       Case 2 &  1.25 & 0.0306  &  $6\frac{2}{3}$  & 6
  &4& 
  814.259   &  $9.170\cdot 10^{-6}$  & 0.002838 & $[6.6601,6.6675]$\\
         Case 2 &  1 & 0.2806  &  $6\frac{2}{3}$  & 6
  &4& 
  702.074   &  $9.715\cdot 10^{-5}$  & 0.003892  & $[6.6588,6.6667]$\\
      Case 3 &  1450 & 19.870  &  $77\frac{1}{2}$  & 4
  &18& 
  57860.670   &  $8.578\cdot10^{-5}$  & 0.1804  & $[77.163,77.477]$ 
  \\
   Case 3 &  1000 & 469.870  &  $77\frac{1}{2}$  & 4
  &18& 
   44259.934   &  $2.651\cdot10^{-3}$  & 0.2324  & $[77.090,77.445]$\\
 \hline
\end{tabular}
\end{center}
\end{table}

Next, we give the same types of results for quadratic holding cost function $h(i)=i^2$.

\begin{table}[H]
\begin{center}
\caption{Lower bounds of $\Delta_c$, thresholds $i^\ast$ and values of $\ell$ and  $T$ for improving policies and corresponding lower bounds of average cost reduction and estimated average cost reductions with corresponding $99\%$ confidence intervals for various instances with quadratic holding costs $h(i)= i^2$ and $\mu=\mu_1$.}
\label{table0sqd}
\begin{tabular}{ |p{1.1cm}|p{0.9cm}|p{1.4cm}|p{0.7cm}|p{0.5cm}|p{0.5cm}|p{1.5cm} |p{1.8cm} |p{1.4cm} |p{2.4cm} | }
 \hline
 \multicolumn{2}{|c|}{Case and $c_r$} & \multicolumn{8}{|c|}{ Results}\\
 \hline
  Case & $c_r$   &  $\Delta_c$ & $g_\mu$ & $i^\ast$ & $\ell$ & $T$ & Lower bound & Avg. cost saved & 99\% Confidence interval\\
 \hline
   Case 1 &  0.045 & 0.005  &  $\frac{18}{25}$  & 4 
  &2& 
  242.510   &   $4.950\cdot10^{-6} $ & 0.04222  & $ [0.6774,0.6782] $\\
     Case 1 &  0.025 & 0.025 &  $\frac{18}{25}$  & 4 
  &2& 
  219.688   &  $2.721\cdot10^{-5} $& 0.04652
  & $ [0.6730,0.6739]$  \\
       Case 2 &  9.8 & 0.115  &  $4\frac{8}{9}$  & 4
  &4& 
  853.948   &   $3.290\cdot 10^{-5}$  & 0.08558  & $[4.7972,4.8094]$\\
         Case 2 &  8 & 1.915  &  $4\frac{8}{9}$  & 4
  &4& 
  711.673   &   $6.543\cdot 10^{-4}$ & 0.1047 & $[4.7776,4.7908]$ \\
      Case 3 &  23300 & 7.669  &  $496$  & 0
  &20& 
  74212.112   &  $2.582\cdot 10^{-5}$  & 2.2509 & $[491.41,496.09]$
 \\
   Case 3 &  20000 & 3307.669  &  $496$  & 0
  &20& 
   48148.535   &  $1.716\cdot 10^{-2}$  & 3.2106 & $[489.89,495.69]$
   \\
 \hline
\end{tabular}
\end{center}
\end{table}

For both holding cost functions and for the different traffic intensities we see similar results; as $\Delta_c$ increases, there is a decrease of $T_1$ and an increase in the average expected cost saved. The lower bound of the average expected cost reduction seems much lower than the actual average cost saved but as $\Delta_c$ increases, the lower bound steeply increases. In many cases the approximate 99\% confidence interval contains the average cost of the system without control. This limits the strength of conclusions that can be drawn from the observed average expected cost reductions.
Finally, we note that the approximate confidence intervals are wider for cases with a higher traffic intensity due to the substantial increase in simulation time resulting from longer renewal times.

\subsection{Saved average cost when $\mu=\mu_2$}

For $\mu=\mu_2$ we give the results again, first for linear holding cost function $h(i)=5\cdot i$, and then for quadratic holding cost function $h(i)=i^2$.

\begin{table}[H]
\begin{center}
\caption{Lower bounds of $\Delta_c$, thresholds $i^\ast$ and values of $\ell,m$ and  $T_m$ for improving policies and corresponding lower bounds of average cost reduction and estimated average cost reductions with corresponding $99\%$ confidence intervals for various instances with linear holding costs $h(i)= 5\cdot i$ and $\mu=\mu_2$.}
\label{table1lin}
\begin{tabular}{ |p{1.05cm}|p{0.5cm}|p{1.0cm}|p{0.5cm}|p{0.2cm}|p{0.2cm}|p{0.7cm}|p{1.5cm} |p{1.8cm} |p{1.5cm} |p{2.8cm} | }
 \hline
 \multicolumn{2}{|c|}{Case and $c_r$} & \multicolumn{9}{|c|}{ Results}\\
 \hline
  Case & $c_r$   &  $\Delta_c$ & $g_\mu$ & $i^\ast$ & $\ell$ & $m$ & $T_m$ & Lower bound & Avg. cost saved & 99\% Confidence interval\\
 \hline
   Case 1 &  94 & 0.910  &  $11 \frac{3}{7}$ & 6 
  &0& 37 &
  295.045   &  $2.486\cdot 10^{-4}$ & 0.02546 & $ [11.4023,11.4040] $ \\
     Case 1 &  80 & 14.910  &  $11\frac{3}{7}$ & 6 
  &0& 29 &
  224.956   &  $5.288 \cdot 10^{-3}$ & 0.2654 & $[11.1621,11.1642]$\\
       Case 2 &  170 & 0.944  &  $15$  & 7
  &1& 75 &
  1070.129   &  $7.216\cdot 10^{-5}$  & 0.02234 & $[14.9754,14.9799]$ \\
         Case 2 &  150 & 20.944  &  $15$  & 7
  &1& 57 &
  798.006  &  $2.134\cdot 10^{-3}$  & 0.1237  & $[14.8737, 14.8789 ]$\\
      Case 3 &  125 & 4.323  &  $61\frac{2}{3}$  & 5
  &6& 1072 &
  90574.437  &  $3.975\cdot10^{-6}$  & 0.009599 & $[61.5811,61.7330]$\\
   Case 3 &  100 & 29.323 &  $61\frac{2}{3}$  & 5
  &6& 1003 &
   83811.938   &  $1.920\cdot10^{-5}$  & 0.005311 & $[61.5821,61.7406]$\\
 \hline
\end{tabular}
\end{center}
\end{table}

\begin{table}[H]
\begin{center}
\caption{Lower bounds of $\Delta_c$, thresholds $i^\ast$ and values of $\ell,m$ and  $T_m$ for improving policies and corresponding lower bounds of average cost reduction and estimated average cost reductions with corresponding $99\%$ confidence intervals for various instances with quadratic holding costs $h(i)= i^2$ and $\mu=\mu_2$.}
\label{table1sqd}
\begin{tabular}{ |p{1.05cm}|p{0.5cm}|p{1.0cm}|p{0.7cm}|p{0.2cm}|p{0.2cm}|p{0.7cm}|p{1.7cm} |p{1.8cm} |p{1.5cm} |p{2.5cm} | }
 \hline
 \multicolumn{2}{|c|}{Case and $c_r$} & \multicolumn{9}{|c|}{ Results}\\
 \hline
  Case & $c_r$   &  $\Delta_c$ & $g_\mu$ & $i^\ast$ & $\ell$ & $m$ & $T_m$ & Lower bound & Avg. cost saved & 99\% Conf. interval\\
 \hline
   Case 1 &  95 & 2.219  &  $10 \frac{22}{49}$ & 5 
  &0& 34 &
 280.835   &  $6.358\cdot 10^{-4}$ & 0.1067  & $[10.341,10.343]$\\
     Case 1 &  80 & 17.219  &  $10\frac{22}{49}$ & 5 
  &0& 29 &
  234.633   &  $5.866\cdot 10^{-3}$ & 0.3656 & $[10.082,10.084]$\\
       Case 2 &  170 & 2.613  &  $13$  & 5
  &1& 71 &
  1061.860   &  $2.013\cdot 10^{-4}$  & 0.1861
  & $[12.811,12.817]$
  \\
         Case 2 &  150 & 22.613  &  $13$  & 5
  &1& 58 &
  867.672   &  $2.123\cdot 10^{-3}$  & 0.3116
  &$[12.685,12.692]$
  \\
      Case 3 &  54 & 0.389  &  $233\frac{8}{9}$  & 1
  & 5 & 1212 &
  107476.358   &  $3.015\cdot10^{-7}$  & 0.2546
  & $[232.89,234.38]$
  \\
   Case 3 &  40 & 14.389 &  $233\frac{8}{9}$  & 1
  &5& 1049 &
   91395.169   &  $1.311\cdot10^{-5}$  & 0.1216
   & $[232.95,234.58]$
   \\
 \hline
\end{tabular}
\end{center}
\end{table}

The conclusions one can draw when $\mu=\mu_2$ are similar to the conclusions for the case $\mu=\mu_1$. For Cases 1 and 2, we note a seemingly steeper increase of the average cost reduction as $\Delta_c$ increases. For Case 3 we do not observe a cost reduction as $\Delta_c$ increases which we deem to result from the high uncertainty in the actual average expected cost reduction.
Additionally we get insight in the decrease of $m$ as $\Delta_c$ increases.

\appendix


\bibliographystyle{amsrefs}
\bibliography{AvgCostRefs}

\appendix

\end{document}